\documentclass[11pt]{article}

\usepackage{dsfont}
\usepackage{color}

\usepackage[top=2 cm,bottom=2 cm,left=2. cm, right=2. cm]{geometry}

\usepackage{amssymb}
\usepackage{graphicx}
\usepackage{amsmath}
\usepackage{amsthm}
\usepackage{mathtools}
\usepackage{amsfonts}

\let\pa\partial  
\let\na\nabla  
\let\eps\varepsilon   
  
\newcommand{\R}{{\mathbb R}} 
\newcommand{\diver}{\operatorname{div}}

\newcommand{\FL}{(- \Delta)^{\alpha}}
\newcommand{\al}{\alpha}

\newcommand{\uh}{{\hat u}}
\newcommand{\zero}{\left\{ 0 \right\}}
\newcommand{\Nn}{\tilde{\mathcal{N}}}

\newcommand{\dd}{\textrm{d}}

\allowdisplaybreaks

\newcommand{\RomanNumeralCaps}[1]
     {\MakeUppercase{\romannumeral #1}}

\newtheorem{theorem}{Theorem}   
\newtheorem{lemma}[theorem]{Lemma}   
   
\newtheorem{remark}{Remark} 
  
\newtheorem{definition}{Definition}

\title{Derivation of a Fractional Cross-Diffusion System as the Limit of a Stochastic Many-Particle System Driven by L\'{e}vy Noise}

\author{Esther S. Daus \thanks{Institute for Analysis and Scientific Computing, Vienna University of Technology, Wiedner Hauptstra\ss e 8--10, 1040 Wien, Austria (esther.daus@tuwien.ac.at)}, \; 
Mariya Ptashnyk \thanks{Department of Mathematics, School of Mathematical and Computer Sciences, Heriot-Watt University, EH14 4AS Edinburgh, Scotland, United Kingdom (m.ptashnyk@hw.ac.uk)}, \; 
Claudia Raithel \thanks{Institute for Analysis and Scientific Computing, Vienna University of Technology, Wiedner Hauptstra\ss e 8--10, 1040 Wien, Austria (claudia.raithel@tuwien.ac.at)}
}

\date{}

\begin{document}

\maketitle

\begin{abstract}
  In this article a fractional cross-diffusion system is derived as the rigorous many-particle limit of a multi-species system of moderately interacting particles that is driven by L\'{e}vy noise. The form of the mutual interaction is motivated by the porous medium equation with fractional potential pressure. Our approach is based on the techniques developed by Oelschl\"ager (1989) and Stevens (2000), in  the latter of  which the convergence of a regularization of the empirical measure to the solution of a correspondingly regularized macroscopic system is shown.  A well-posedness result and the non-negativity of solutions are proved for the regularized macroscopic system, which then yields the same results for the non-regularized fractional cross-diffusion system in the limit. 
\end{abstract}





\section{Introduction} 

Cross-diffusion systems arise in modelling many different biological and physical processes, e.g.~the movement of cells, bacteria or animals; transport through ion-channels in cells; tumour growth; gas dynamics; carrier transport in semiconductors~\cite{BWS12, GP_1984, JB_2002, KS_1971_2, Painter_2009,  R_95, SKT_79, WK_2000}, with the chemotaxis system~\cite{KS_1971} being one of the most important examples of a cross-diffusion system (with a triangular cross-diffusion matrix). Different approaches, ranging from semigroup theory to energy or entropy methods and  applications of the
Jordan-Kinderlehrer-Otto scheme, have been used to analyze cross-diffusion systems~\cite{AL83, A_1989, A_1990,   BDiFFS_2018,  CFSS_2018,  DLMT15, DiFF_2013, Jue16, Lady68, Le_2002, SW_2021, Tuoc_2007}, with many results dedicated to the chemotaxis model in particular, see the review papers \cite{BBTW_2015, Horstman_2003, Horstman_2004} and the references therein. Cross-diffusion equations with nonlocal interaction terms have also attracted interest in previous years~\cite{BBP17, DiFF_2013,   FEF18}.  

The derivation of cross-diffusion systems from stochastic $N$-particle systems has been studied in \cite{O_89}, assuming some ellipticity of the cross-diffusion matrix. A new approach, using a regularized system and an intermediate ``frozen'' system, was necessary for the rigorous derivation of a chemotaxis system from a microscopic description of stochastic particle interactions \cite{S_2000}. Some models of cross-diffusion type used in population dynamics were derived in \cite{CDJ_2018,MF_2015}, whereas in \cite{Seo_2018} the Maxwell-Stefan equations were obtained as the hydrodynamic limit of the empirical densities. 
In \cite{FES20} a cross-diffusion model with nonlocal interactions was derived from a many-particle system with a Newtonian potential. 
%

Although there are several contributions concerned with nonlocal cross-diffusion systems available, the derivation of cross-diffusion systems with fractional cross-diffusion terms from the stochastic particle systems, as well as well-posedness results for such systems have not been considered.  Correspondingly, the aim of this article is to rigorously derive such a system starting from the microscopic model, the movement of the particles being determined by a L\'{e}vy walk and non-local mutual interaction potentials. In particular, we first derive a fractional cross-diffusion system as the many-particle limit of a moderately interacting particle system and then we prove a well-posedness result for the limiting system.

The motivation for considering processes driven by L\'evy walks is derived from the experimental observation that both in the context of cell motility \cite{EGP18, EGP19, HBC12, KEV04, MJD09, TG05, VBH99} and population dynamics \cite{BPP03, VAB96, RMM04,  RSH11} in certain situations organisms move according to L\'{e}vy processes. Especially in the absence of an attractant \cite{EGP18} or  when targets are rare and can be visited any number of times \cite{VBH99}, the distribution of runs asymptotically behaves like an inverse square  power-law distribution leading to L\'{e}vy walks as optimal movement and search strategies.
L\'{e}vy walks were also used in modelling human mobility \cite{RSH11} and swarm robotic systems \cite{RG20}, see also \cite{ZDK15} for an overview. 

In this work we derive the following fractional cross-diffusion system:
\begin{equation}\label{sys.final}
\begin{aligned} 
& \pa_t u_i + \sigma_i(- \Delta)^{\alpha}u_i - \diver\Big( \sum_{j=1}^n a_{ij} u_i \nabla^{\beta} u_j\Big)  = 0 \; \; \; \;  &&\text{in } \; (0,T) \times \R^d, \\
 & u_i(0, \cdot) = u^0_{i}  \; \; \; \; && \text{in } \; \R^d, \quad i=1,\ldots, n,
 \end{aligned} 
\end{equation}
for $T>0$ with $a_{ij} \in \R $ and $\sigma_i>0$. Here $\nabla^\beta u_j := \nabla ( (-\Delta)^{\frac{\beta - 1} 2} u_j)$  and  we consider $\alpha \in (1/2,1)$ and $\beta \in (0,1)$ in such a way that $2\alpha > \beta +1$, meaning  that self-diffusion dominates cross-diffusion effects. This restriction is necessary in our derivation of the limiting result as well as in proving the well-posedness of the cross-diffusion system \eqref{sys.final}. 


The form of the non-local interaction in the fractional cross-diffusion term in \eqref{sys.final} is motivated by the porous medium equation with fractional potential pressure that has been treated by Caffarelli and V\'{a}zquez, see \cite{CV_2011, CV_2011_2, CSV_2013} and the overview \cite{V_2017}. Their equation is, in particular, given by $v_t = \nabla \cdot (v \nabla p(v) ),$ where the pressure $p(v) =  (-\Delta)^{-s}v$ for $s \in (0,1)$. This model has appeared in the context of the macroscopic evolution and the phase segregation dynamics of particles systems with short- and long-range interactions  \cite{GL_97, GL_98,GLM_2000}. It, furthermore, appears in the study of dislocations \cite{BKM_2010, H_72}.

The starting point of our analysis is the microscopic description of the particle dynamics, which will be introduced in detail in Section \ref{subsec.1.1}. It is given in terms of a system of SDEs  assuming that there are $n$ species, each with $N_i$ particles for $i = 1, \ldots, n$. In our model, the dynamics are influenced by two forces: a nonlocal mutual interaction between the subpopulations, which scales in a moderate way as the particle number increases, and random dispersal, which is modelled by $\sum_{i=1}^n N_i$ i.i.d.~L\'{e}vy processes. For simplicity, we assume that the i.i.d.~L\'{e}vy processes are taken to correspond to the fractional Laplacian (in the sense of \eqref{frac.L} below), which then appears in \eqref{sys.final}. However, as in the derivation in \cite{SO_2017}, we expect that our analysis holds for any $2\alpha$-stable L\'{e}vy processes.

In the limiting procedure we use the methods developed by Oelschl\"{a}ger~\cite{ O_89} and Stevens~\cite{S_2000}. The article \cite{O_89} is part of a series of works by the author on this subject (see also \cite{O_85, O_87, O_90}), the first of which drew some inspiration from the previous work \cite{CP_83}, where a propagation of chaos result for the Burgers' equation is proven. The propagation of chaos result contained in \cite{O_85} was then generalized by M\'{e}l\'{e}ard and Roelly-Coppoletta in~\cite{MRC_87}. Furthermore, in \cite{HL16} propagation of chaos is shown for a Keller-Segel system with fractional diffusion. The main technique in \cite{O_89, S_2000} and which we also use here is to, using It\^{o}'s formula and martingale estimates, examine the asymptotic behaviour of a regularization of the empirical measure, now viewed as a stochastic process taking values in $L^2(\R^d)^n$. The novelty in our analysis lies in the structure of the fractional cross-diffusion terms, whose handling requires some new technical ingredients.

The limiting procedure that we use relies on the existence and regularity of solutions to the system \eqref{sys.final}. These issues are addressed in the final two theorems of this paper. While the proofs are quite involved, the main ideas that we use are classical and rely on the Banach fixed-point theorem and higher-order \textit{a priori} estimates. Due to the fractional nature of \eqref{sys.final},  in our arguments we require the use of the fractional Leibniz rule and Gagliardo-Nirenberg inequality. For the reader's convenience, any results concerning fractional Sobolev spaces that are needed in our proofs are listed in the Appendix.

The structure of this paper is as follows: We first introduce our microscopic model and review some standard facts about L\'{e}vy processes. In Section~\ref{ref_2}, we formulate the main results. Then, in Sections~\ref{ref_3} and~\ref{ref_4} we give the arguments for our convergence results. In Sections~\ref{ref_5} and~\ref{ref_6} we prove existence and uniqueness of non-negative solutions for the limiting macroscopic model. 


\subsection{Description of the microscopic dynamics}\label{subsec.1.1} We consider the following system of $\sum_{i=1}^n N_i$ SDEs:
 \begin{align}\label{micro.sys}
  dX_i^{k,N}(t) = - \sum_{j=1}^n \frac{1}{N}\sum_{\ell=1}^{N_j} a_{ij} \nabla^{\beta} \hat V_N\big(X_{i}^{k,N}(t)-X_{j}^{\ell,N}(t) \big)dt + \sqrt{2 \sigma_i} \, d L_i^k(t), 
 \end{align}
for $i=1,\ldots, n$ and $k = 1,\ldots, N_i$, with $a_{ij} \in \R$ and $\sigma_i>0$. Here, $X_i^{k,N}(t)$ denotes the position of the $k$-th particle of species $i$ at time $t>0$ and the $L_i^k$  are i.i.d.~L\'{e}vy processes corresponding to the fractional Laplacian.

The interaction potential that we use is $(-\Delta)^{\frac{\beta-1}{2}}\hat V_N$ for $\beta \in (0,1)$. Here, $\hat V_N$ is defined in terms of a radially symmetric probability density $W_1$ as
\begin{align}\label{kernel_1}
\hat V_N := W_N * \hat W_N \; \textrm{ for } \,  W_N(x) = \kappa_N^d W_1(\kappa_N x) \, \textrm{ and } \, \hat W_N(x) = \hat \kappa_N^d W_1(\hat \kappa_N x), 
\end{align}
where $\kappa_N = N^{\kappa/d}$ and $\hat \kappa_N = N^{\hat \kappa/d}$ for exponents $\kappa$ and $\hat \kappa$ that satisfy conditions given in \eqref{scaling} and $\kappa > \hat{\kappa}$. The properties satisfied by $W_1$ are listed in \eqref{assumption_1}-\eqref{assumption_3}.

In order for our limiting theorems to hold, it is important that the scaling of the interaction is \textit{moderate}. In particular, we consider an interaction to be ``moderate'' if, in the many-particle limit, the mutual interaction does not depend on the microscopic fluctuations of the particle densities.  To verify that our interaction is moderate we perform a heuristic calculation, similar to \cite{O_85}: Assume for simplicity that the processes $X_{i}^{k,N}(t)$ for $i =1,\ldots,n$ and $k = 1,\ldots,N_i$ are i.i.d.~with a smooth density $\mu(t, \cdot)$ and, furthermore, that each $N_i = N$. We consider the variance of the force exerted at $x \in \R^d$, which is given by 
\begin{align*}
\begin{split}
& J:=\textrm{Var} \Big(\frac{1}{N} \sum_{j=1}^n  \sum_{k=1}^{N_j} \nabla^{\beta}  \hat V_N\big(x-X_{j}^{k,N}(t) \big)\Big) \\
&\leq \frac{C}{N} \Big[  \int_{\R^d}   \big|\nabla^{\beta} \hat V_N\big(x- y\big)\big|^2  \mu(t,y)  \,\dd y -  \Big( \big(\hat V_N   \ast \nabla^{\beta} \mu(t, \cdot)\big) (x)\Big)^2\Big].
\end{split}
\end{align*}
We treat the first term on the right-hand side of the above expression using
\begin{align}
\begin{split}
\label{moderate_1}
 &\int_{\R^d} \big|\nabla^{\beta}  \hat V_N\big(x- y\big)\big|^2  \mu(t,y) \, \dd y = \int_{\R^d} \big|\nabla^{\beta}(\hat W_N\ast W_N)(x-y)\big|^2  \mu(t,y)  \dd y\\
& \quad = \int_{\R^d}  \kappa_N^{d + 2 \beta} \hat \kappa_N^{2d} \,  \big|(\hat W_1 (\hat \kappa_N \cdot ) \ast  \nabla^{\beta}W_1 (\kappa_N \cdot )) (s) \big|^2  \mu(t,x + \kappa_N^{-1} s)   \dd s,
\end{split}
\end{align}
where we have made the change of variables $s =  \kappa_N (y-x)$.  We notice that
\begin{align*}
\begin{split}
\Big| ( \hat W_1 (\hat \kappa_N \cdot ) \ast \nabla^{\beta} W_1 (\kappa_N \cdot )) (s)  \Big|
& \leq \int_{\R^d}  \big|\hat W_1 (\hat \kappa_N z)  \nabla^{\beta} W_1 (s-\kappa_N z) \big| \, \dd z \\
& =  \kappa_N^{-d} \int_{\R^d}\Big| \hat W_1 \Big(\frac{\hat \kappa_N}{\kappa_N} s^{\prime}\Big) \Big| \, \big|  \nabla^{\beta}W_1 (s- s^{\prime}) \big|    \, \dd s^{\prime},\\
\end{split}
\end{align*}
where $s^{\prime} = \kappa_N z$. Plugging this into \eqref{moderate_1} and using that $ \kappa > \hat  \kappa$ yields that 
\begin{align*}
J \lesssim N^{-1} \kappa_N^{-d+ 2 \beta} \hat \kappa_N^{2d}  \leq N^{-1} \kappa_N^{d + 2\beta}  \rightarrow 0 \quad \text{ as } N\rightarrow \infty, \quad \text{ when } \kappa \text{ satisfies \eqref{scaling}}.
\end{align*}

\subsection{Regularized empirical processes}
The empirical processes $S^N_i(t)$ corresponding to the subpopulations are given by 
\begin{align*}
S^{N}_i(t):=\frac{1}{N} \sum_{k=1}^{N_i} \delta_{X_i^{k,N}(t)}, \qquad \langle S^{N}_i(t), \psi \rangle = \frac{1}{N} \sum_{k=1}^{N_i} \psi ( X_i^{k,N}(t))
\end{align*}
for $i =1,\ldots,n$ and any real-valued function $\psi$ on $\mathbb{R}^d$. Throughout this paper, for any real-valued measure $\nu$, we use the notation
  $$  \langle \nu , \psi \rangle  := \int_{\R^d}  \psi(x) \nu( dx). $$  

In Theorem \ref{Convergence}, we show that certain regularizations of the empirical processes converge to the solution of a regularized version of \eqref{sys.final}. We introduce the following regularized versions of the empirical processes:
\begin{equation}\label{def_densities}
\begin{aligned}
 \hat s_i^N(t,x)&:=\big(S_i^N(t)*\hat V_N\big)(x),\qquad 
  h_i^N(t,x)&:=\big(S_i^N(t)*W_N\big)(x),
\end{aligned}
\end{equation}
where we use the notation from \eqref{kernel_1}.  With \eqref{def_densities} we are able to rewrite the system \eqref{micro.sys} as
 \begin{equation}\label{micro.sys.new}
   dX_i^{k,N}(t) = - \sum_{j=1}^n a_{ij} \nabla^{\beta} \hat{s}^{N}_j\big(t,X_{i}^{k,N}(t)\big)dt + \sqrt{2 \sigma_i}\,  d L_i^k(t),  
 \end{equation}
 for $k=1, \ldots, N_i$ and $i=1, \ldots, n$.

\subsection{It\^{o}'s formula for L\'{e}vy processes} \label{Ito_intro} 
For $i =1, \ldots, n$ and $k= 1,\dots, N_i$, the $L^k_i(t)$ in \eqref{micro.sys.new} are i.i.d.~L\'{e}vy processes on a filtered probability space $(\Omega, \mathcal{F}, \mathcal{F}_t,\mathbb{P})$ corresponding to $(-\Delta)^{\alpha}$. We mean this in the sense that the L\'{e}vy measure $\nu$ of the processes is given by 
$$
d \nu :=  \frac{c_{d,\alpha}}{|z|^{d+2\alpha}} \dd z,
$$
where $1/2<\alpha<1$ and $c_{d,\alpha}$  is a dimensional constant that is, e.g., given in \cite[Section 3]{NPV_2012}.  With $\nu$ defined as above, for any real-valued function $\psi$ with sufficient regularity, the nonlocal operator $\mathcal{L}$ corresponding to the $L^k_i(t)$ satisfies
\begin{align}
\begin{split} \label{frac.L}
\mathcal{L}\psi & := \int_{\R^d} \big( \psi(x + z) - \psi(x) - \nabla \psi(x) \cdot z \chi_{|z|\le 1}\big) \, \dd \nu(z)\\
&\qquad \  = -c_{d,\alpha} \, \textrm{P.V.}\int_{\R^d}\frac{\psi(x) - \psi(y)}{|x-y|^{d+2\al}}\, \dd y
 =: -(-\Delta)^{\alpha} \psi,
\end{split}
\end{align}
where P.V.~denotes the Cauchy principal value. 

As it is the main tool of our derivation, we now give It\^{o}'s formula for the dynamics determined by~\eqref{micro.sys.new}. The natural space of test functions is given by 
\begin{align*}
C_b^{1,2\alpha}( \R_+ \times \R^d) = \Big\{ \psi \in C^{1,1}_b(  \R_+ \times \R^d) \, | \,  \, (-\Delta)^{\alpha} \psi \in C^0_b(\R_+\times \R^d) \Big\},
\end{align*}
where $C^0_b(\R_+ \times \R^d)$ is the space of continuous bounded functions and $C^{1,1}_b(  \R_+ \times \R^d)$ also requires continuous and bounded  derivatives with respect to time and space. For $\psi \in C_b^{1,2\alpha}( \R_+ \times \R^d) $ the dynamics given by \eqref{micro.sys.new} then yield that
\begin{align}\label{ito}
\begin{split}
& \langle S_i^N(t),\psi(t, \cdot)\rangle  =  \langle S_i^N(0),\psi(0, \cdot) \rangle  
-  \sum_{j=1}^n\int_0^t \big\langle S_i^N(\tau), a_{ij} \na^{\beta} \hat s_j^N\big(\tau, X_i^{k,N}(\tau)\big) \cdot \na \psi(\tau, \cdot)\big\rangle \,  \dd \tau 
\\& -\sigma_i \int_0^t \big\langle S_i^N(\tau),(-\Delta)^{\alpha}  \psi(\tau, \cdot) \big\rangle  \, \dd \tau
+ \frac{1}{N} \sum_{k=1}^{N_i} \int_0^t \int_{\R^d\setminus\{0\}} \hspace{-0.2 cm }   \sqrt{2\sigma_i} \, D_z\psi\big(\tau, X_i^{k,N}(\tau_-)\big) \,  \Nn_i^k( \dd z \dd \tau).
\end{split}
\end{align}
 Here, $X_i^{k,N}(\tau_-)$ denotes the one-sided limit of $X_i^{k,N}(t)$ as $t  \nearrow \tau$ and 
$$
 D_z f(y):= f(y+z) -f(y) \quad \textrm{for any} \quad z,y \in \R^d.
$$
Furthermore, the compensated Poisson measure $\Nn_i^k$ is defined by 
$$
\Nn_i^k((0, t]\times U) := \mathcal N_i^k((0, t]\times U) - t \nu(U) \quad \textrm{for any } \quad U \in \mathcal B(\R^d\setminus \{0 \})  \textrm{ and } t>0,
$$
where $\mathcal N_i^k$ is the Poisson measure 
\begin{align*}
\mathcal{N}_i^k((0, t]\times U) := \sum_{\tau \in (0, t]} \textbf{1}_U(L^k_{i}(\tau) - L^k_{i}(\tau_-)) .
\end{align*}
The above expression is a sum because it can be shown that a.s.~the L\'{e}vy process has only a finite number of jumps in a bounded interval. For the reader's convenience, we remark that a useful reference on L\'{e}vy processes is \cite{A_2004}.


 \subsection{Additional notation} \label{notation} Unless otherwise stated, we use the convention that the indices $i,j = 1,\ldots,n $ denote species, whereas $k, \ell = 1,\ldots,N_i$ are used to denote the $k$-th (or $\ell$-th) particle.
 
We will use $\| \cdot \|_p$ to denote $\| \cdot \|_{L^p(\R^d)}$ for $p \in (1, \infty]$. Furthermore, for $\alpha \in (0,1)$ and $p \in (1, \infty]$ we use $\| \cdot \|_{W^{\alpha, p}}$ to denote $\| \cdot \|_{W^{\alpha, p}(\R^d)}$ and similarly $\| \cdot \|_{H^{\alpha}}$ denotes $\| \cdot \|_{H^{\alpha}(\R^d)}$. For $T>0$, we denote the natural norm associated with \eqref{sys.final} on $(0,T) \times\R^d$:
\begin{align}
\label{natural_norm}
 \|f\|^2_{[0,T]} := \sup_{0 \leq t \leq T}\|f(t)\|_{2}^2 + \int_0^T \|( - \Delta)^{\frac \alpha 2} f(t)\|_{2}^2 \, \dd t.
\end{align}

As in \cite{O_89, S_2000}, for two positive finite real-valued measures $\nu_1, \nu_2 \in \mathcal{M}(\R^d)$, we consider
 $$
 d(\nu_1, \nu_2) := \sup\Big\{ \big\langle \nu_1 - \nu_2, \psi \big\rangle \; | \;   \psi \in C^1_b(\R^d), \; \; \|\psi\|_{L^{\infty}(\R^d)} + \| \nabla \psi \|_{L^{\infty}(\R^d)}\le 1 \Big\}.
 $$
Throughout the article, we denote $\hat{u}^N= (\hat{u}^N_1, \ldots, \hat{u}^N_n)$  and 
$
\| \hat{u}^N \|^2_2 =  \sum_{i=1}^n \| \hat{u}^N_i \|_2^2 ,
$
 analogous notation is used for all other $n$-dimensional vectors (e.g.~$u$, $h^N$, $s^N$, and $\hat{s}^N$) and other norms.
 We use the notation $``\lesssim"$ in order to denote $``\leq C(n, \alpha, \beta, a_{ij}, \sigma_i, d)"$. If there are additional dependencies for the universal constant, e.g.~on a time $T>0$, then we write $``\lesssim_T"$. Often the universal constant may not depend on the full retinue of $n, \alpha, \beta, d, a_{ij},$ and $\sigma_i$, but we still use the notation $``\lesssim"$.

 \section{Formulation of the main results} \label{ref_2}
   
We have already defined $\hat V_N,$ $W_N$, and $\hat W_N$ in terms of $\kappa_N = N^{\kappa/d}$ and $\hat \kappa_N = N^{\hat \kappa/d}$ in \eqref{kernel_1}.  Now, we give the conditions on $\kappa$ and $\hat \kappa$. For a given arbitrarily small $\rho>0$, we require that
\begin{align}\label{scaling}
0<\hat \kappa < \frac{\delta d}{d+4} \qquad \textrm{and} \qquad \delta (1 + \rho)d < \kappa < \frac{d}{d+3},
\end{align}
for some $\delta \in (0,1)$. These conditions are essential for the limiting argument in Theorem \ref{Convergence}.  We shall also use the notation
\begin{equation}\label{delta}
\delta_N := N^{-\delta}.
\end{equation}
We assume the following  properties satisfied by $W_1$:
\begin{align}
&  F(W_1)  \in C^2_b(\R^d), \label{assumption_1}\\
& |F(W_1)(\xi)|  \lesssim  \exp(-C^{\prime} |\xi|),  \label{assumption_2}\\
& | \Delta F(W_1)(\xi)| \lesssim (1 + |\xi|^2) |F(W_1)(\xi) |,  \label{assumption_3}
 \end{align}
where $F$ denotes the Fourier transform and $C^{\prime}>0$ is a constant. We remark that the conditions \eqref{scaling} are similar to those given in (1) of \cite{S_2000}. We, furthermore, mention that the conditions \eqref{assumption_1}-\eqref{assumption_3} are likewise similar to (6)-(8) in \cite{S_2000} and (3.2)-(3.4) in \cite{O_89}, where in both \cite{S_2000} and \cite{O_89} the authors include an additional assumption concerning the decay of $W_1$ along rays.

The first theorem of this paper is a convergence result that shows that a certain regularization of the empirical measure, namely $h^N$ defined in \eqref{def_densities}, converges to $\hat{u}^N$ solving 
\begin{equation}\label{syst:reg}
\begin{aligned} 
& \pa_t \uh_i^N + \sigma_i (-\Delta)^{\alpha}\uh_i^N - \diver \Big(\sum_{j=1}^n  a_{ij} \uh_i^N \na^{\beta}\big(\uh_j^N * \hat W_N\big)\Big)=0    &&\text{ in } \; (0,T) \times \R^d,  \\
& \uh_i^N(0, \cdot) = u^0_{i}  && \text{ in } \; \R^d,  
 \end{aligned} 
\end{equation}
for $i = 1,\ldots, n $ and $T>0$. 
The convergence result is as follows:
\begin{theorem}\label{Convergence}
Let $\alpha \in ( 1/2,1)$ and $\beta \in (0,1)$ satisfy $\beta+1 < 2 \alpha$ and, furthermore, when $d=1$ that $\alpha  - \beta < 1/2$ or $\alpha < 3/4$. The kernel $W_1$ satisfies \eqref{assumption_1}-\eqref{assumption_3}. Assume that $u^0 \in H^s(\R^d)^n$, for $s> d/2 +2$, is non-negative and satisfies 
\begin{align}
&   \lim_{m \to \infty} \sup_{N \in \mathbb N} \mathbb{P} \left[ \sum_{i=1}^n\langle S_i^N(0),  1 \rangle  \geq m\right] = 0, 
\label{initial_condition_assumption_1} \\
&  \lim_{N \to \infty}\mathbb{P} \left[\|h^N(0,\cdot)- u^0\|_{2}^2 \geq \delta_N^{1+\rho}\right] = 0, \label{initial_condition_assumption_2} 
 \end{align}
where $\delta$ and $\rho$ satisfy \eqref{scaling} and we use the notation \eqref{delta}. Then, we have 
 \begin{equation*}
      \lim_{N \to \infty}\mathbb{P} \left[\|h^N -  \uh^N\|^2_{[0,T]}\geq\delta_N\right] = 0,
 \end{equation*}
 where $\hat{u}^N$ solves \eqref{syst:reg}. 
\end{theorem}

We make a couple of remarks concerning the above theorem:

\begin{remark}[Initial condition]\label{rem.2}
\textnormal{Notice that the assumptions \eqref{initial_condition_assumption_1} and \eqref{initial_condition_assumption_2} ensure that $N_i$, which is the number of  particles of species $i$, is of the same order of magnitude as the scaling parameter $N$, i.e.~$N_i \approx N$. An example of an admissible initial condition would be to have $N$ i.i.d.~random variables for species $i$ with distribution $u_i^0 / \| u_i^0\|_{1}$ for $i = 1, \ldots, n$ (see \cite{O_87}).}
\end{remark}

\begin{remark}[Regularization] \label{scaling_why}
\textnormal{In the formulation and proof of Theorem~\ref{Convergence} our use of the regularized problem~\eqref{syst:reg} is similar to \cite[Theorem 6.2]{S_2000}. As we will see in Step 1, the different scalings of the kernels  $W_N$ and $\hat W_N$ are required to obtain uniform (in $N$) boundedness of $\sup_{0\leq t \leq t^N} \| \hat s^N(t)\|_{C^2(\R^d)}$, where $t^N$ is an appropriately defined stopping time. In \cite{S_2000} the analogous estimate is (28), 
whereas in \cite{O_89} 
some ellipticity condition on the cross-diffusion term is used to estimate corresponding terms.
}
 \end{remark}
 
 \begin{remark}[Dominating self-diffusion] \label{dom_self} \textnormal{The restriction on $\beta$, i.e.\ $\beta+1 < 2 \alpha$, including $\beta< \alpha$ for $\alpha < 1$, implies that the self-diffusion dominates the cross-diffusion. The main place we use this assumption is in the well-posedness and regularity results for problems~\eqref{sys.final} and~\eqref{syst:reg}. In particular, we highlight the derivation of the higher-order \textit{a priori} estimates, uniform in $N$, for $\hat{u}^N$ in Theorem~\ref{existence_regular}, which are used in (27). Since our cross-diffusion matrix is not assumed to be triangular, the condition on $\beta$ is also used in Step~3.1 of the proof of Theorem~\ref{Convergence}. Throughout the paper we consider $\alpha<1$, if $\alpha =1$ then many of our calculations could be simplified. We remark that in \cite{S_2000} a cross-diffusion system with triangular cross-diffusion matrix and $\alpha = \beta =1$ is analyzed. 
 }
 \end{remark}
 
In our second theorem, we post-process the result of Theorem \ref{Convergence} in order to compare the not regularized objects, the empirical processes $S^N_i$ and  $u_i$ solving~\eqref{sys.final}. 
\begin{theorem}\label{Convergence_2}
 Assume that the conditions of Theorem~\ref{Convergence} are satisfied and that 
\begin{align}\label{assum_init_2}
\sum_{i=1}^n\langle  u_{i}^0, \psi \rangle \leq C \qquad \textrm{and} \qquad 
 \lim_{m \to \infty} \sup_{N \in \mathbb N} \mathbb{P} \Big[\sum_{i=1}^n \langle S_i^N(0), \psi^2 \rangle  \geq m\Big] = 0, 
\end{align}
 where $C$ is a constant and $\psi(x) = \log(2+ x^2)$, then
 \begin{equation*}
      \lim_{N \to \infty}\mathbb{P} \Big[ \sum_{i=1}^n\sup\limits_{0\le t\le T} d( S_i^N(t), u_i(t))  \ge \mu \Big] = 0
\end{equation*}
for any $\mu >0$. 
 \end{theorem}

Our final two theorems are well-posedness and regularity results that are used in Theorems~\ref{Convergence} and~\ref{Convergence_2}. In Theorem~\ref{existence_regular} we ensure that the system \eqref{syst:reg} has a unique non-negative solution with sufficient regularity. Then, in Theorem \ref{full_system}, we pass to the limit in the regularization to obtain a solution of \eqref{sys.final}.

\begin{theorem} \label{existence_regular}
Assume that the conditions of Theorem~\ref{Convergence} are satisfied. Letting $u^0 \in H^s(\R^d)^n$, for $s > d/2$, be non-negative, the following results hold:
\begin{itemize}
\item[i)] (Local solution)\quad  There exists a time $T= T(\| u^0\|_{H^s(\R^d)})>0$ such that there is a unique non-negative weak solution $\uh^N \in L^\infty(0,T; H^s(\R^d))^n$ of the regularized problem~\eqref{syst:reg} in the time interval $[0,T]$. This solution satisfies
\begin{equation}\label{estim_Hs_1}
 \| \uh^{N}\|_{L^\infty(0,T; H^s(\R^d))} + \| \uh^{N}\|_{L^2(0,T; H^{s+\alpha}(\R^d))}  \leq C 
\end{equation}
and if additionally $s> d/2 + 2$, then we  obtain
\begin{equation}\label{estim_hu_N}
\sup_{(0,T)\times \R^d} | D^2 \uh^N_i (t,x) | \leq C, \quad i = 1, \ldots, n,
\end{equation}
where $C = C(d, \sigma_i, a_{ij}, n)$ is independent of $N$. 
\item[ii)] (Global solution for small initial data) \quad 
Additionally, there exists a constant $\theta = \theta(d, \sigma_i, a_{ij}, n)>0$ such that if
\begin{align}
\label{small_ic}
\| u^0\|_{H^s(\R^d)} \leq \theta(d, \sigma_i, a_{ij},n),
\end{align}
then part $i)$ holds for any $T >0 $.
\end{itemize}
\end{theorem}

Passing to the limit $N \rightarrow \infty$ in the result of Theorem \ref{existence_regular}, we obtain a solution for the original system \eqref{sys.final}. In particular, we find that

\begin{theorem}\label{full_system}
Under the assumptions of Theorem~\ref{existence_regular}, there exists a unique non-negative solution $u$ of problem~\eqref{sys.final} in $L^\infty(0,T; H^s(\R^d))^n \cap L^2(0,T; H^{s+\alpha}(\R^d))^n$ with 
\begin{equation}\label{converg_hatu_u}
 \lim_{N \to \infty} \|\hat u^N - u\|_{[0,T]}^2 = 0.
\end{equation}
Here $T>0$ corresponds to either the local or global existence interval from Theorem~\ref{existence_regular}.
\end{theorem}

\section{Argument for Theorem \ref{Convergence}} \label{ref_3}

The following lemma, which is taken from \cite{O_89}, is the motivation for many of the assumptions on the convolution kernel $W_1$ and is used in the proofs of the main results. 
\begin{lemma}[Lemma 1 of \cite{O_89}] \label{conv_2} 
Assume that $W_1$ satisfies \eqref{assumption_1}-\eqref{assumption_3} and $W_N$ is given in \eqref{kernel_1}. Then, using the convention $U(\cdot ) = W_N(\cdot) | \cdot  |$ and for any $\varepsilon>0$ and $\tau>0$, we have
   \begin{align}   \label{estim_WX_1}
   \begin{split}
\| S_i^N \ast U \|_2^2  \leq C(d) \big[   \kappa_N^{2\varepsilon - 2} \|S_i^N(\tau) \ast W_N\|^2_{2} +   \langle S_i^N(\tau) , 1\rangle^2 \exp{( - C^{\prime} \kappa_N^{\varepsilon})}\big], 
  \end{split}
 \end{align}
 for $i = 1, \ldots, n$.\\
For $f \in H^1(\R^d)$ we have that
\begin{equation} \label{converg_convol} 
\| f\ast \hat W_N - f \|_2^2  \leq C(d) \hat \kappa_N^{-2} \|\nabla f \|_2^2. 
\end{equation} 
\end{lemma} 
\noindent 
Since there is no birth or death  in our dynamics, $\langle S_i^N(\tau) , 1\rangle = N_i / N$ for all $\tau \in (0,T]$.

For the proof of Lemma \ref{conv_2} we refer to \cite{O_89}. Here, we only remark that the proof relies on properties of  the Fourier transform and exploits the assumptions \eqref{assumption_1}-\eqref{assumption_3}.

\subsection{Proof of Theorem \ref{Convergence}}

In the proof of Theorem~\ref{Convergence} we follow ideas from \cite[Theorem 1]{O_89} and \cite[Theorem 6.2]{S_2000}. 
%
%
The novelty of our proof lies mainly in technical issues that we encounter due to the form of the nonlocal cross-diffusion terms. Some of these issues can be easily remedied by using the fractional Leibniz rule or the fractional Gagliardo-Nirenberg inequality (see Appendix). The main new contribution is the estimate contained in Step~3.1. While the majority of our proof quite closely follows~\cite[Theorem 6.2]{S_2000}, we give the full argument for the sake of completeness.

\begin{proof}
Our argument proceeds in four steps: 

\noindent \textbf{Step 1: Introduction of a stopping time.}  We introduce a first hitting time $t^N$ such that 
\begin{align}\label{stopping_time}
t^N  = t^N (\omega) := \inf\big\{ \tau>0 \, \big \vert  \, \| h^N - \hat u^N\|^2_{[0,\tau]} (\omega) >\delta_N \big\} \quad \textrm{for} \quad \omega \in \Omega.
\end{align}
Assumptions \eqref{initial_condition_assumption_1} and \eqref{initial_condition_assumption_2}, together with 
the right-continuity of $\| h^N - \hat u^N\|^2_{[0,\tau]}$, for $0\leq \tau \leq T$, ensure that the limit of $t^N$, as $N\to \infty$, is positive a.s.~in~$\Omega$ (see Appendix). 
In addition,  the right-continuity of $\| h^N - \hat u^N\|^2_{[0, \tau]}$, for $0\leq \tau \leq  T$,  yields that $t^N$ is a stopping time and 
\begin{align}
\begin{split}
\label{prob_fact}
\mathbb{P} \big[ \| h^N - \hat u^N\|^2_{[0, t^N\wedge T]}  \geq \delta_N \big] = 
\mathbb{P} \big[\| h^N - \hat u^N\|^2_{[0, T]}  \geq \delta_N \big].
\end{split}
\end{align}
Let $k$ be a multi-index. Using the Cauchy-Schwarz inequality, the definition of $t^N$ and that of $\hat{W}_N$ in \eqref{kernel_1},  and the assumption \eqref{assumption_2} on $W_1$ we obtain
\begin{equation*}
 \begin{aligned}
 &\sup\limits_{x\in \R^d} \big| D^k \big [\hat s^N(t,x) - (\hat{u}^N( t, \cdot) \ast \hat W_N )(x)\big] \big| \\
 &\quad \leq \| h^N(t, \cdot) - \hat u^N( t, \cdot)\|_2 \sup\limits_{x\in \R^d} \|D^k \hat W_N (x - \cdot) \|_2
\leq \sqrt{\delta_N} \hat \kappa_N^{|k|+ \frac{d}{2}}\| D^k W_1\|_{2} \lesssim  \sqrt{\delta_N} \hat \kappa_N^{|k|+ \frac{d}{2}},
 \end{aligned}
\end{equation*}
for $0\leq t\leq t^N$.
By our assumptions on $\hat \kappa_N$ and $\delta_N$, see \eqref{scaling} and \eqref{delta}, we have that
$$
  \sqrt{\delta_N} \hat \kappa_N^{2+ \frac{d}{2}}=   N^{-\frac{\delta}{2}} N^{\frac{\hat \kappa}{d}(2+\frac{d}{2})} \leq 1 \quad \text{ for }  \quad N\geq 1.
$$
Then using the triangle inequality and \eqref{estim_hu_N} of Theorem~\ref{existence_regular} yields
\begin{align}\label{estim_s_N}
\begin{split}
&  \sup_{0\leq t \leq t^N \wedge T} \| \hat s^N(t)\|_{C^2(\R^d) } 
\\&
 \leq  \sup_{0\leq t \leq t^N\wedge T} \Big(   \sum_{|k| \leq 2}\sup\limits_{x\in \R^d} \big|D^k \big[\hat s^N(t,x) - (\hat{u}^N(t, \cdot) \ast \hat W_N )(x)\big] \big| 
 +  \| \hat u^N(t) \|_{C^2(\R^d)}  \Big) \lesssim 1.
\end{split}
\end{align} 
%

\noindent\textbf{Step 2: Deriving an expression for $\| h^N - \uh^N \|_2^2$ .} \quad For $i = 1, \dots, n$, we apply It\^{o}'s formula~\eqref{ito} to compute directly the expressions for $\langle h_i^N , h_i^N  \rangle$ and  $\langle h_i^N , \uh^N_i  \rangle$. The calculations for $\langle h_i^N , h_i^N  \rangle$ and  $\langle h_i^N , \uh^N_i  \rangle$ are similar to those in \cite{O_89, S_2000}, however for completeness we include here the main steps of the derivation of the equation for $\| h^N - \uh^N \|_2^2$. Let $t\in (0, t_N \wedge T]$.
\smallskip 

\noindent \textbf{Step 2.1:} Starting with $\langle h_i^N , h_i^N  \rangle$, by  \eqref{def_densities} we obtain
\begin{align*}
 \langle h_i^N(t, \cdot), h_i^N(t, \cdot) \rangle = \frac{1}{N^2}\sum_{k,\ell=1}^{N_i} V_N\big(X_i^{k,N}(t) -X_i^{\ell,N}(t)\big),
\end{align*}
where $V_N := W_N \ast W_N$. Then we use the equation for $X_i^{k,N} - X_i^{\ell,N}$ obtained from \eqref{micro.sys.new}, that the L\'{e}vy processes $L^k_i$ are i.i.d., and that $\na V_N$  and $D_z V_N$ are odd for any $z \in \R^d$, to write 
\begin{align*}
 & \langle h_i^N(t, \cdot), h_i^N (t, \cdot)\rangle =  \frac{1}{N^2}\sum_{k,\ell=1}^{N_i} V_N\big(X_i^{k,N}(0)-X_i^{\ell,N}(0)\big)\\
  &  - \frac{2}{N^2}\sum_{j=1}^n \sum_{k,\ell=1, k\neq \ell}^{N_i} a_{ij}\int_0^t  \na^{\beta} \hat s_j^N\big(\tau, X_i^{k,N}(\tau)\big) \cdot \na V_N\big(X_i^{k,N}(\tau) - X_i^{\ell,N}(\tau)\big)\, \dd \tau \\
&- \frac{2}{N^2} \sigma_i \sum_{k,\ell=1,k\neq \ell}^{N_i}\int_0^t (-\Delta)^{\alpha} V_N\big(X_i^{k,N}(\tau) - X_i^{\ell,N}(\tau)\big)\, \dd \tau \\
&+\frac{2}{N^2} \sum_{k,\ell=1, k\neq \ell}^{N_i} \sqrt{2 \sigma_i} \int_0^t \int_{\R^d\setminus\{0\}}D_z V_N\big(X_i^{k,N}(\tau_{-}) - X_i^{\ell,N}(\tau_-)\big) \, \Nn_i^k(\dd z \dd\tau).
\end{align*}
\textbf{Step 2.2:}  For $\langle h_i^N, \uh^N_i \rangle$, we use the definition of $h^N_i$ to obtain
\begin{align}\label{hu_1}
\langle h_i^N(t, \cdot), \uh^N_i(t, \cdot) \rangle =  \int_{\mathbb{R}^d} \uh_i^N(t, x ) \frac{1}{N} \sum_{k=1}^{N_i} W_N\big(X_i^{k,N}(t) -x \big) \, \textrm{d}x.
\end{align}
 Making use of the relation 
$$
v(t) \int_0^t g(\tau) \, \dd \tau = \int_0^t \partial_\tau \Big[ v(\tau) \int_0^\tau g(\xi) \, \dd \xi \Big] \, \dd \tau
$$
in conjunction with It\^{o}'s formula, we can write 
\begin{equation}\label{hu_22}
\begin{aligned}
&  \langle h_i^N(t, \cdot), \uh^N_i(t, \cdot) \rangle = \langle h_i^N(0,\cdot), \uh^N_i(t,\cdot) \rangle \\
 &- \frac1N \int_{\R^d}\int_0^t \uh^N_i(\tau,x) \sum_{k=1}^{N_i} \sum_{j=1}^n a_{ij}\na^{\beta}\hat s_j^N\big(\tau, X_i^{k,N}(\tau)\big) \cdot \na W_N\big(X_i^{k,N}(\tau)-x\big)  \, \dd \tau \dd x \\
 &- \frac{\sigma_i}{N} \int_{\R^d}\int_0^t \uh^N_i(\tau,x) \sum_{k=1}^{N_i} (-\Delta)^{\alpha} W_N\big(X_i^{k,N}(\tau)-x\big)\, \dd \tau \dd x \\
 &+ \frac{\sqrt{2\sigma_i}}{N} \int_{\R^d}\int_0^t \uh^N_i(\tau,x)\sum_{k=1}^{N_i}  \int_{\R^d\setminus\{0\}} D_z W_N\big(X_i^{k,N}(\tau_-)-x\big)\, \Nn_i^k(\dd z \dd\tau) \,  \dd x \\
 &+ \frac{1}{N} \int_{\R^d}\int_0^t \pa_{\tau}\uh^N_i(\tau,x)\sum_{k=1}^{N_i} \Big(W_N\big(X_i^{k,N}(\tau)-x\big) - W_N\big(X_i^{k,N}(0)-x\big)\Big)\, \dd\tau \dd x. 
\end{aligned}
\end{equation}
We then use
$$
 \frac{1}{N}  \int_{\R^d} \int_0^t   \pa_{\tau}\uh^N_i(\tau,x)  \, \dd\tau \sum_{k=1}^{N_i} W_N\big(X_i^{k,N}(0)-x\big) \, \dd x 
 = \langle h_i^N(0,\cdot), \hat u_i^N(t,\cdot) -  \uh_i^N(0,\cdot) \rangle
$$
and the system \eqref{syst:reg} for $\uh^N_i$ to rewrite the last term of \eqref{hu_22} as
\begin{align}
\label{new_edit_1}
\begin{aligned}
&   \langle h^N_i(0,\cdot), \uh_i^N(0,\cdot) \rangle
-\langle h_i^N(0,\cdot), \hat u_i^N(t,\cdot)\rangle  - \sigma_i \int_0^t \big\langle (-\Delta)^{\frac \alpha 2}h_i^N(\tau,\cdot), (-\Delta)^{\frac \alpha 2}\uh_i^N (\tau,\cdot)\big\rangle \, \dd\tau\\
&  - \int_0^t\big \langle \na (-\Delta)^{\frac {\alpha-1} 2 } h_i^N(\tau, \cdot), \sum_{j=1}^n a_{ij} (-\Delta)^{\frac {1-\alpha} 2}\big(\uh_i^N(\tau,\cdot) \na^{\beta} (\uh_j^N*\hat W_N)(\tau,\cdot)\big)\big\rangle \,  \dd\tau.
\end{aligned}
\end{align}
Notice that in the above computation we have used \eqref{integration_by_parts} from the Appendix. 
Plugging the identity \eqref{new_edit_1} into \eqref{hu_22} implies 
\begin{align*}
 &  \langle h_i^N(t, \cdot), \uh^N_i(t,\cdot) \rangle = \langle h^N_i(0,\cdot), \uh_i^N(0,\cdot) \rangle  - \int_0^t \big\langle S_i^N(\tau), \sum_{j=1}^n a_{ij}\na^{\beta} \hat s_j^N(\tau,\cdot)\cdot \na (\uh^N_i*W_N) (\tau,\cdot)\big\rangle \, \dd\tau \\
 &\quad -\sigma_i\int_0^t \big\langle S_i^N(\tau),  \FL(\uh^N_i*W_N) (\tau,\cdot) \big\rangle \, \dd\tau - \sigma_i \int_0^t \big\langle (-\Delta)^{\frac \alpha 2}h_i^N(\tau,\cdot), (-\Delta)^{\frac \alpha 2}\uh_i^N (\tau,\cdot)\big\rangle \, \dd\tau \\
 &\quad+\frac{\sqrt{2\sigma_i}}{N} \sum_{k=1}^{N_i}\int_0^t \int_{\R^d\setminus\{0\}} D_z\big(\uh^N_i*W_N\big)\big(\tau, X_i^{k,N}(\tau_{-})\big)\, \Nn_i^k( \dd z \dd\tau) 
 \\
 &\quad- \int_0^t\big \langle \na (-\Delta)^{\frac {\alpha-1} 2 } h_i^N(\tau, \cdot), \sum_{j=1}^n a_{ij} (-\Delta)^{\frac {1-\alpha} 2}\big(\uh_i^N(\tau,\cdot) \na^{\beta} (\uh_j^N*\hat W_N)(\tau,\cdot)\big)\big\rangle \, \dd\tau .
\end{align*}
%
Considering $\uh^N_i $ as a test function in  \eqref{syst:reg} and integrating by parts  yields the equation for $\langle \uh_i^N(t,\cdot), \uh_i^N(t,\cdot)\rangle$.

\smallskip

\noindent\textbf{Step 2.3:} Combining the previous calculations, we obtain
\begin{align}
 & \|h^N(t,\cdot)-\uh^N(t,\cdot)\|_2^2
 =  \|h^N(0,\cdot)-\uh^N(0,\cdot)\|_2^2 \tag{\RomanNumeralCaps{1}} \label{term1}\\
 &- \sum_{i,j=1}^n 2 a_{ij}\int_0^t \left\langle S_i^N(\tau), \na^{\beta}\hat s^N_j(\tau,\cdot)\cdot \na\Big( \big(h_i^N- \uh_i^N\big)*W_N(\tau,\cdot) \Big)  \right\rangle \, \dd\tau \tag{\RomanNumeralCaps{2}} \label{term2} \\
 &+ \sum_{i,j=1}^n 2a_{ij}\int_0^t \left\langle \na (-\Delta)^{\frac {\alpha-1} 2 }\big(h_i^N(\tau,\cdot) - \uh_i^N(\tau,\cdot)\big), \right. \tag{\RomanNumeralCaps{3}} \label{term4} \\
 & \hspace{4.5cm} \left.  (-\Delta)^{\frac{1- \alpha} 2}\Big( \uh_i^N  (\tau,\cdot) \na^{\beta}\big(\uh_j^N*\hat W_N (\tau,\cdot)\big) \Big)\right\rangle \, \dd\tau  \nonumber  \\
  &- \sum_{i=1}^n 2\sigma_i\int_0^t \left\langle S_i^N(\tau), \FL \Big(\big(h_i^N- \uh_i^N\big)*W_N(\tau,\cdot)\Big) \right\rangle \, \dd\tau \tag{\RomanNumeralCaps{4}} \label{term3} \\
 &+ \sum_{i=1}^n2\sigma_i \int_0^t \left\langle (-\Delta)^{\frac \alpha 2}\big(h_i^N(\tau,\cdot) - \uh_i^N(\tau,\cdot)\big), (-\Delta)^{\frac \alpha 2} \uh_i^N (\tau,\cdot) \right \rangle \, \dd\tau  \tag{\RomanNumeralCaps{5}} \label{term5}   \\
 &+ \sum_{i=1}^n\frac 2 N \FL V_N(0) \int_0^t \big\langle S_i^N(\tau), \sigma_i \big\rangle \, d\tau  + \sum_{i=1}^n M^N_i(t)\tag{\RomanNumeralCaps{6})+(\RomanNumeralCaps{7}} \label{term6}.
\end{align}
Here, we have used the notation 
\begin{align*}
& M^N_i(t) :=    \frac{\sqrt{8 \sigma_i}}{N}    \sum_{k=1}^{N_i}    \int_0^t  \int_{\R^d \setminus \zero}  \hspace{-0.1 cm } D_{z}\big( \big[ \big(h^N_i(\tau_-, \cdot) - \uh^N_i (\tau, \cdot)\big)  \ast W_N \big] \big( X_i^{k,N}(\tau_{-})  \big) \big) \,  \Nn_i^k (\dd z \dd \tau).
\end{align*}

\noindent \textbf{Step 3: Estimates for terms \eqref{term2}-{\rm(\RomanNumeralCaps{7})}.}

\smallskip

\noindent\textbf{Step 3.1: Terms \eqref{term2} + \eqref{term4}.}  We write \eqref{term4} = (\RomanNumeralCaps{3}.1) + (\RomanNumeralCaps{3}.2) + (\RomanNumeralCaps{3}.3), where 
\begin{align*}
(\textrm{\RomanNumeralCaps{3}}.1)& = \hspace{-0.2 cm} \sum_{i,j=1}^n  2a_{ij} \int_0^t\left \langle  \na (-\Delta)^{\frac {\alpha-1} 2} \left(h_i^N(\tau)-\uh^N_i(\tau)\right), \right.\\
 & \hspace{4.2cm} \left.(-\Delta)^{\frac{1- \alpha} 2} \Big(\uh^N_i(\tau) \na^{\beta}\big( \uh^N_j (\tau) \ast \hat{W}_N  - \hat s_j^N(\tau) \big)\Big)\right \rangle \dd\tau, \\
(\textrm{\RomanNumeralCaps{3}}.2) & =  \hspace{-0.2 cm}\sum_{i,j=1}^n 2a_{ij}\int_0^t \Big \langle \nabla (-\Delta)^{\frac{\alpha-1} 2} \big(h_i^N (\tau) - \uh^N_i(\tau)\big), 
(-\Delta)^{ \frac{1- \alpha} 2}  \Big[\big(\uh^N_i (\tau) - h_i^N(\tau)\big)\na^{\beta}\hat s_j^N(\tau) \Big]\Big \rangle  \dd\tau,  \\
(\textrm{\RomanNumeralCaps{3}}.3) & = \hspace{-0.2 cm} \sum_{i,j=1}^n 2a_{ij}\int_0^{t} \left \langle \nabla  (-\Delta)^{\frac {\alpha-1} 2 }  \big(h_i^N (\tau) - \uh^N_i(\tau)\big),  (-\Delta)^{\frac{1- \alpha} 2} \big( h_i^N(\tau)\na^{\beta}\hat s_j^N(\tau) \big)\right \rangle \dd\tau. 
\end{align*}
Then we obtain
\begin{align*} 
\begin{split}
& |\textrm{(\ref{term2})} + \textrm{(\RomanNumeralCaps{3}.3)}|  
=  \Big| \sum_{i,j=1}^n  2a_{ij} \int_0^t \int_{\R^d} \Big\langle S_i^N(\tau),    (-\Delta)^{\frac{1- \alpha} 2}R^N_j  (\tau, \cdot,y)     \nabla  (-\Delta)^{\frac {\alpha-1} 2 }  G_i^N(\tau,y) \Big\rangle  \dd y \dd\tau\Big| \\
& \lesssim  \sum_{i,j=1}^n \int_0^t \Big[ C_{\varsigma} \hspace{-0.05 cm}  \int_{\R^d} \hspace{-0.05 cm} \Big| \frac{1}{N} \sum_{k=1}^{N_i}  (-\Delta)^{\frac{1-\alpha}2} R^N_j\big(\tau, X^{k,N}_i(\tau),y\big) \Big|^2 \dd y 
 + \varsigma \big \| (-\Delta)^{\frac{\alpha}2} \big( h_i^N(\tau) -\uh^N_i(\tau)\big)\big\|^2_{2} \Big]   \dd \tau,  
\end{split}
\end{align*}
where $G_i^N(\tau,y)= h_i^N(\tau,y) -\uh^N_i(\tau,y)$ and $R^N_j(\tau, x,y) =W_N(x-y) \big(\nabla^{\beta} \hat s_j^N(\tau,x) -\nabla^{\beta} \hat s_j^N(\tau,y)  \big)$.  
For $\tau \in (0,t)$ and arbitrary $\varepsilon>0$, we process the first term on the right-hand side of the last inequality using Parseval's identity as
$$
\begin{aligned}
& \int_{\R^d}  \Big| \frac{1}{N} \sum_{k=1}^{N_i} (-\Delta)^{\frac{1-\alpha}2} R^N_j\big(\tau, X^{k,N}_i(\tau),y\big) \Big|^2 \dd y  
 = 
  \int_{\R^d}  \hspace{-0.1 cm }  |\xi|^{2(1-\alpha)} \Big| F\Big(\frac{1}{N} \sum_{k=1}^{N_i}  R^N_j\big(\tau, X^{k,N}_i(\tau),\cdot\big) \Big) (\xi) \Big|^2 \dd \xi   
\\
& \quad  \leq 
\int_{|\xi|\leq \kappa_N^{1+\varepsilon}}  \hspace{-0.1 cm } |\xi|^{2(1-\alpha)} \Big| F\Big(\frac{1}{N} \sum_{k=1}^{N_i}  R^N_j\big(\tau, X^{k,N}_i(\tau),\cdot\big) \Big) (\xi)  \Big|^2 \dd\xi  
\\ 
& \qquad + 
\int_{|\xi| > \kappa_N^{1+\varepsilon}}   \hspace{-0.1 cm }  |\xi|^{2(1-\alpha)} \Big| F\Big(\frac{1}{N} \sum_{k=1}^{N_i}  R^N_j\big(\tau, X^{k,N}_i(\tau),\cdot\big) \Big) (\xi) \Big|^2 \dd \xi =: I_1 + I_2. 
\end{aligned}
$$
Similar to \cite{O_89,S_2000}, we treat  $I_1$ using Parseval's identity, the bound  \eqref{estim_s_N}, and \eqref{estim_WX_1} of Lemma~\ref{conv_2}:
\begin{align*} 
I_1 
& \leq     \kappa_N^{2 (1-\alpha)(1+ \varepsilon)}  \| \hat s_j^N(\tau) \|^2_{C^2} \int_{\R^d}\Big( \frac 1 N  \sum_{k=1}^{N_i}   W_N\big(X^{k,N}_i (\tau)-y\big)\big|X^{k,N}_i (\tau) - y\big| \Big)^2  \dd y\\
& \lesssim   \kappa_N^{2 (1-\alpha)(1+ \varepsilon)}  \kappa_N^{2\varepsilon - 2}   \|S_i^N(\tau) \ast W_N\|^2_{2} +   \Big(\frac{N_i}{N}\Big)^2 \exp{( - C^{\prime} \kappa_N^{\varepsilon})}.
\end{align*} 
%
Due to the fractional derivative, we need to use a different approach than in \cite{O_89,S_2000} to handle~$I_2$. For this we first split it into two parts:
\begin{align*} 
\begin{split}
I_2 \leq&  \int_{|\xi| > \kappa_N^{1+\varepsilon}} |\xi|^{2(1-\alpha)} \Big| F\Big( \frac 1 N \sum_{k=1}^{N_i}  W_N(X^{k,N}_i(\tau)-\cdot) \nabla^{\beta} \hat s_j^N(\tau,X^{k,N}_i(\tau))\Big) (\xi)\Big|^2 \dd \xi
\\
&  + 
 \int_{|\xi| > \kappa_N^{1+\varepsilon}} |\xi|^{2(1-\alpha)} \Big| F\Big( \frac 1 N \sum_{k=1}^{N_i}  W_N(X^{k,N}_i(\tau)-\cdot)
\nabla^{\beta} \hat s_j^N(\tau,\cdot)  \Big)(\xi) \Big|^2 \dd\xi =: J_1 + J_2.
\end{split}
\end{align*}
The term $J_1$ can be treated using standard properties of the Fourier transform, Jensen's inequality for sums,  assumption \eqref{assumption_2}, and  estimate \eqref{estim_s_N}. In particular, we find that
\begin{align*}
\begin{split}
&J_1   =   \int_{|\xi| > \kappa_N^{1+\varepsilon}} |\xi|^{2(1-\alpha)} \Big|
\frac 1N  \sum_{k=1}^{N_i}  \nabla^{\beta}  \hat s_j^N(\tau, X^{k,N}_i(\tau)) F\big(W_N(X^{k,N}_i(\tau)-\cdot)\big)(\xi)\Big|^2 \dd \xi 
\\ 
&\leq  \|  \hat s_j^N (\tau) \|^2_{C^1}  \frac{N_i}{N} \int_{|\xi| > \kappa_N^{1+\varepsilon}} |\xi|^{2(1-\alpha)}   \frac 1N  \sum_{k=1}^{N_i} 
\Big| F\big( \delta_{X^{k,N}_i(\tau) } \ast W_N \big)(\xi)\Big|^2 \dd \xi  
\\ 
& \lesssim  \Big( \frac{N_i}{N} \Big)^2  \int_{|\xi^{\prime}| > \kappa_N^{\varepsilon}} |\xi^\prime|^{2(1-\alpha)}  \kappa_N^{2(1-\alpha)+d}  \exp{(-2 C^\prime |\xi^\prime|} ) \, \dd \xi^\prime \lesssim \Big( \frac{N_i}{N} \Big)^2  \exp{(-  C^{\prime} \kappa_N^{\varepsilon}} ) .
\end{split}
\end{align*}
To treat $J_2$, we once more split it into a near-field and far-field contribution, but now corresponding to the integral coming from an additional convolution that turns up as
\begin{align*}
& F\Big( \frac 1 N \sum_{k=1}^{N_i}  W_N(X^{k,N}_i(\tau)-\cdot)
\nabla^{\beta} \hat s_j^N(\tau,\cdot)  \Big) (\xi) \\
& = \int_{\R^d}  F\Big( \frac 1 N \sum_{k=1}^N  W_N(X^{k,N}_i(\tau)-\cdot)\Big)(\xi - \eta)
 F \big(\nabla^{\beta} \hat s_j^N(\tau, \cdot)  \big) (\eta) \, \dd\eta .
\end{align*}
Applying the triangle inequality then yields
\begin{align*}
\begin{split}
J_2& \leq   \int_{|\xi| > \kappa_N^{1+\varepsilon}}   |\xi|^{2(1-\alpha)}  \Big| \int_{|\eta|\leq \kappa_N^{1+\varepsilon}}   F\big(S_i^N (\tau)\ast W_N\big)(\xi - \eta) F \big(\nabla^{\beta} \hat s_j^N(\tau)  \big) (\eta)  \, \dd\eta \Big|^2 \dd\xi 
\\
&  + 
 \int_{|\xi| > \kappa_N^{1+\varepsilon}}  |\xi|^{2(1-\alpha)} \Big| \int_{|\eta|> \kappa_N^{1+\varepsilon} }    F\big( S_i^N (\tau)\ast W_N \big)(\xi - \eta)
 F \big(\nabla^{\beta} \hat s_j^N(\tau)  \big) (\eta)  \, \dd\eta \Big|^2 \dd\xi  =: K_1 + K_2.
\end{split}
\end{align*}
The term $K_1$ can be estimated using the  properties of the Fourier transform  along with the assumption \eqref{assumption_2} and another application of Jensen's inequality for sums. We additionally make use of $| \xi - \eta| + | \eta | \geq |\xi |$ for $\xi, \eta \in \R^d$. Using these tools yields
\begin{align*}
K_1 
& \leq \Big( \frac{N_i}{N}\Big)^2 \int_{|\xi| > \kappa_N^{1+\varepsilon}}    |\xi|^{2(1-\alpha)}  \Big | \int_{|\eta|\leq  \kappa_N^{1+\varepsilon}}     |F\big(  W_N \big)(\xi - \eta)| | F \big(\nabla^{\beta} \hat s_j^N(\tau)  \big) (\eta)|  \, \dd \eta \Big|^2 \dd\xi  
\\ 
&\leq  \Big( \frac{N_i}{N}\Big)^2  \int_{|\xi| > \kappa_N^{1+\varepsilon}}  |\xi|^{2(1-\alpha)}    
  \\& \hspace{1.5cm} \times 
\Big|  \int_{|\eta|\leq  \kappa_N^{1+\varepsilon}} |\eta|^{\beta}  |F\big(  W_N \big)(\xi - \eta)\big|| F (S^N_j)(\eta) F(W_N)(\eta) F (\hat W_N) (\eta) \big| \dd\eta \Big|^2  \dd\xi
\\
& 
\leq  \Big( \frac{N_j}{N}\Big)^2  \Big( \frac{N_i}{N}\Big)^2 \int_{|\xi| > \kappa_N^{1+\varepsilon}}  |\xi|^{2(1-\alpha)}   \kappa_N^{2 \beta(1+\varepsilon)} \\
& \hspace{1.5cm} \times  \Big | \int_{|\eta|\leq  \kappa_N^{1+\varepsilon}}  \Big| F\big(  W_1 \big)\Big( \frac{\xi - \eta}{\kappa_N} \Big)\Big|
\Big| F(W_1)\Big(\frac{\eta}{\kappa_N} \Big)\Big| \Big| F (W_1) \Big(\frac{\eta}{\hat \kappa_N} \Big) \Big| \, \dd\eta \Big|^2  \dd\xi  
\\
&  \lesssim  \Big( \frac{N_j}{N}\Big)^2  \Big( \frac{N_i}{N}\Big)^2   \int_{|\xi| > \kappa_N^{1+\varepsilon}}  |\xi|^{2(1-\alpha)}   \kappa_N^{2\beta(1+\varepsilon)}  \Big | \int_{|\eta|\leq  \kappa_N^{1+\varepsilon}} 
\exp{\Big(- C^\prime \Big( \frac{|\xi|}{\kappa_N} + \frac{|\eta|}{\hat\kappa_N}\Big)\Big)}
\, \dd\eta  \Big|^2 \dd \xi 
\\
&  \lesssim \Big( \frac{N_j}{N}\Big)^2  \Big( \frac{N_i}{N}\Big)^2   \int_{|\xi^\prime| > \kappa_N^{\varepsilon}} \hspace{-0.2 cm} |\xi^\prime|^{2(1-\alpha)} \kappa_N^{2\beta(1+\varepsilon)} \kappa_N^{3d + 2(1- \alpha)} \hspace{-0.05 cm }   
\exp{\big(-2 C^\prime |\xi^\prime| \big)}  \dd\xi^\prime\\
 &\lesssim \Big( \frac{N_j}{N}\Big)^2  \Big( \frac{N_i}{N}\Big)^2 \exp{\big(-  C^{\prime} \kappa_N^{\varepsilon}\big)}  .
\end{align*}
Using similar methods as above, we write
\begin{align*}
K_2 
& \leq  \Big( \frac{N_j}{N}\Big)^2   \Big( \frac{N_i}{N}\Big)^2  \int_{|\xi^\prime| > \kappa_N^{\varepsilon}}   |\xi^{\prime}|^{2(1-\alpha)}   \kappa_N^{2(1-\alpha) +  2\beta +  3d }  \exp{\big(- 2 C^\prime |\xi^{\prime}| \big) }  \dd\xi^{\prime}  \\
& \hspace{2 cm} \times \Big |  \int_{|\eta^\prime|>  \kappa_N^{\varepsilon}} |\eta^\prime|^{\beta}
 \exp{\Big( -C^\prime |\eta^\prime| \frac{{\kappa_N}}{\hat\kappa_N} \Big)}  \dd\eta^\prime \Big |^2  \lesssim  \Big( \frac{N_j}{N}\Big)^2\Big( \frac{N_i}{N}\Big)^2 \exp{ \big( -  C^{\prime} \kappa_N^{\varepsilon} \big)}.
\end{align*}
Here $\xi^{\prime} = \xi/ \kappa_N$ and $\eta^{\prime} = \eta/ \kappa_N$.
Compiling the above estimates, we find that
\begin{align}\label{edit_2_1}
\begin{split}
  & \int_{\R^d}\Big| \frac{1}{N} \sum_{k=1}^{N_i}  (-\Delta)^{\frac{1-\alpha}2} R^N_j\big(\tau, X^{k,N}_i(\tau),y\big) \Big|^2 \dd y  
  \\& \quad   \quad \quad   
  \lesssim   \kappa_N^{2(1- \alpha) (1 + \varepsilon)} \kappa_N^{2\eps-2} \|h_i^N(\tau, \cdot)\|^2_{2} +   \Big[ \Big( \frac{N_j}{N}\Big)^2 +1 \Big] \Big( \frac{N_i}{N}  \Big)^2  \exp{( - C^{\prime} \kappa_N^{\varepsilon})}.
\end{split}
\end{align}
Summing in \eqref{edit_2_1} over $i, j = 1,\ldots,n$   and  using \eqref{estim_Hs_1} of Theorem~\ref{existence_regular} yields
\begin{align*}
\begin{split}
 | \textrm{(\ref{term2})} + \textrm{(\RomanNumeralCaps{3}.3)}| & \le  \varsigma \int_0^t  \big\| (- \Delta)^{\frac{\alpha}{2}} \big( h^N(\tau,\cdot) -   \uh^N(\tau,\cdot)\big)\big\|^2_{2} \,\dd\tau
\\
 &\quad + C_\varsigma \Big[ \kappa_N^{2(1- \alpha) (1 + \varepsilon)}  \kappa_N^{2\varepsilon -2} \int_0^t \Big(\|  (h^N - \uh^N)(\tau,\cdot) \|^2_2 +1  \Big)  \dd \tau \\
 & \qquad \qquad +  \sum_{i,j = 1}^n \Big[ \Big( \frac{N_j}{N}\Big)^2 +1 \Big]  \Big( \frac{N_i}{N}  \Big)^2  \exp{( - C^{\prime} \kappa_N^{\varepsilon})} t \Big].
\end{split}
\end{align*}
To estimate \textrm{(\RomanNumeralCaps{3}.1)},  we use \eqref{equivalence}, \eqref{frac_Leib}, and \eqref{GN_form} of the Appendix and \eqref{estim_Hs_1} of Theorem \ref{existence_regular}:
\begin{align*}
\begin{split}
 | \textrm{(\RomanNumeralCaps{3}.1)} |
 &  \le \varsigma \int_0^t \big\|(-\Delta)^{\frac{\alpha}{2}} \big(\uh^N(\tau ,\cdot) - h^N(\tau, \cdot)\big) \big\|^2_{2} \, \dd\tau 
 \\&\quad 
 + C_\varsigma \int_0^t  \|  \uh^N(\tau, \cdot) \|^2_{H^{s+1 - \alpha}}    \big \| \na^{\beta} \big(\uh^N (\tau, \cdot) - h^N(\tau, \cdot)\big) \big  \|^2_{H^{1-\alpha}}    \dd \tau
\\ & \le \int_0^t  \big( C_{\varsigma^{\prime}} 
 \|\uh^N(\tau, \cdot)- h^N(\tau, \cdot)\|_{2}^2 + \varsigma^{\prime} \|(-\Delta)^{\frac{\alpha}{2}}(\uh^N(\tau, \cdot) - h^N(\tau, \cdot))\|_{2}^2 \big) \, \dd\tau, 
\end{split}
\end{align*}
for any $\varsigma$ and $\varsigma^{\prime}>0$. Notice that we have used $0 < 1- \alpha + \beta < \alpha$. Our treatment of \textrm{(\RomanNumeralCaps{3}.2)} follows along the same lines, but we replace the use of \eqref{frac_Leib} by that of \eqref{frac_Leib_2} and \eqref{estim_Hs_1} by \eqref{estim_s_N}: 
\begin{align*}
|\textrm{(\RomanNumeralCaps{3}.2)}| \le
   \int_0^t \Big(C_\varsigma  \|\uh^N(\tau, \cdot) - h^N(\tau, \cdot)\|_{2}^2 + \varsigma \big\|(-\Delta)^{\frac \alpha 2}\big(\uh^N(\tau, \cdot) -h^N(\tau, \cdot)\big)\big\|_{2}^2 \Big) \dd\tau,
\end{align*}
for any $\varsigma>0$ and where we have used that $1-\alpha + \beta <2$ to apply \eqref{estim_s_N}.
\smallskip

\noindent \textbf{Step 3.2:  Terms \ \eqref{term3},  \eqref{term5}, and~{\rm(\RomanNumeralCaps{6})}.} The sum of the terms (\ref{term3}) and (\ref{term5}) satisfies
\begin{align*}
\textrm{(\ref{term3})} + \textrm{ (\ref{term5})} \lesssim - \int_0^t\big\| (- \Delta)^{\frac \alpha 2} \big( \uh^N (\tau,\cdot) - h^N(\tau,\cdot)\big) \big \|_{2}^2 \, \dd\tau.
\end{align*}
For~{\rm(\RomanNumeralCaps{6}}), using that $\langle S_i^N, \sigma_i \rangle \lesssim N_i/N$, we find that 
\begin{align*}
|{\rm(\RomanNumeralCaps{6})}| \lesssim \frac{1}{N} \sum_{i=1}^n \frac{N_i}{N} \kappa_N^{d+2\alpha} t.
\end{align*}
 \textbf{Step 3.3: Compilation of the estimates. } Combining the estimates from Steps 3.1~and~3.2 and choosing $\varsigma, \varsigma^{\prime}>0$ small enough, we obtain for $0 < \tilde{T} \leq T$:
\begin{equation}\label{estim_sup_norm}
\begin{aligned}
 &\hspace{-.2cm} \sup_{0\leq t\leq \tilde T \wedge t_N} \|h^N(t,\cdot)-\uh^N(t,\cdot)\|_{2}^2   + \int_0^{\tilde T \wedge t_N}   \big\| (-\Delta)^{\frac\alpha 2} ( h^N  -    \hat{u}^N)(\tau,\cdot)\big\|^2_{2}\,  \dd \tau \\
 &\;  \lesssim   \|h^N(0,\cdot)-\uh^N(0,\cdot)\|_{2}^2 +
 \int_0^{\tilde T \wedge t_N}  \sup_{0\leq\xi \leq \tau} \| h^N(\xi,\cdot) - \uh^N(\xi,\cdot)\|_{2}^2  \,  \dd\tau  \\
 & \;\; + \kappa_N^{4 \varepsilon - 2\alpha (1 + \varepsilon)}  \int_0^{\tilde T \wedge t_N} \hspace{-0.1 cm } \Big(  \sup_{0\leq\xi \leq \tau}\|  h^N (\xi,\cdot) - \uh^N(\xi,\cdot) \|^2_2 + 1 \Big)  \dd \tau  + \frac{  \kappa_N^{d+2\alpha}}{N} \sum_{i=1}^n  \frac{N_i}{N}  \tilde T 
 \\
 & \; \; 
    +   \sum_{i,j=1}^n  \Big[\Big(\frac{N_j}{N}\Big)^2 +1 \Big] \Big( \frac{N_i}{N}  \Big)^2 \exp{( - C^{\prime} \kappa_N^{\varepsilon})} \tilde T+    \sum_{i=1}^n \sup_{0\leq t\leq \tilde T \wedge t_N}  \big|M_i^N(t)\big|.
\end{aligned}
\end{equation}

\noindent \textbf{Step 3.4: Estimate for the martingale term~{\rm(\RomanNumeralCaps{7})}. }  First notice that 
\begin{align}
\label{Jensens_2}
 \mathbb{E}  \Big[  \sum_{i=1}^n \sup_{0\leq t \leq  \tilde{T} \wedge t^N }  \big|M^N_i (t)\big|  \, \Big | \mathcal{F}_0  \Big]^2
 \lesssim \sum_{i=1}^n \mathbb{E}  \Big[ \sup_{0\leq t \leq  \tilde{T} \wedge t^N }  \big|M^N_i (t)\big| \,  \Big | \mathcal{F}_0  \Big]^2,
\end{align}
since the $L^k_i$ are i.i.d. To treat the right-hand side, we begin by noting that, due to the optional sampling theorem, the stopped process $M^N_i(t \wedge t^N)$ is a martingale.  We can then  apply Jensen's inequality and Doob's $L^p$- martingale inequality and use the mutual independence of the $L^{k}_i$  to write
\begin{equation} \label{martingale_1} 
\begin{aligned}
 & \mathbb{E}  \Big[  \sup_{0 \leq t \leq  \tilde{T} \wedge t^N }  \big|M^N_i (t)\big|\,   \Big | \mathcal{F}_0  \Big]^2  \leq
 \mathbb{E}  \Big[   \sup_{0\leq t \leq  \tilde{T}}  \big|M^N_i (t \wedge t^N )\big|^2 \,  \Big|  \mathcal{F}_0   \Big] \leq 4\mathbb{E}  \Big[  \big|M^N_i (\tilde{T} \wedge t^N)\big|^2\,   \Big|  \mathcal{F}_0  \Big]
 \\
& \lesssim   \frac{ 1}{N}  \mathbb{E}  \Big[ \frac{ 1}{N}  \sum_{k=1}^{N_i}   \Big|  \int_0^{\tilde{T} \wedge t^N}\hspace{-0.2 cm }  \int_{\R^d \setminus \zero} \hspace{-0.1 cm} D_{z}\big( \big[ G^N_i(\tau_-, \tau, \cdot)  \ast W_N \big] ( X_i^{k,N}(\tau_{-})  \big) \big) \Nn_i^k (\dd z \dd \tau) \Big|^2 \, \Big|  \mathcal{F}_0  \Big], 
\end{aligned} 
\end{equation}
where $G^N_i(\tau_-, \tau, x) = h^N_i(\tau_-, x) - \uh^N_i (\tau, x)$. 
We continue by using the It\^{o} isometry (see \cite[Chapter 4]{A_2004}), in conjunction with the observation that the jump-set of a L\'{e}vy process is a Lebesgue null set, which means that within the time integral we may replace the left limit $h^N_i(\tau_-, \cdot)$ by $h^N_i(\tau, \cdot)$. Finishing-off the estimate with an application of Jensen's inequality with respect to the measure determined by the density $W_N$, we obtain
\begin{equation*}
\begin{aligned}
& \mathbb{E}  \Big[  \sup_{0\leq t \leq  \tilde{T} \wedge t^N }  \big|M^N_i (t)\big| \,  \Big|  \mathcal{F}_0  \Big]^2\\
& \lesssim 
\frac{ 1}{N}   \mathbb{E}  \Big[    \int_0^{\tilde{T} \wedge t^N} 
 \Big\langle S^N_i(\tau, \cdot),  \int_{\R^d \setminus \zero} \big| D_{z}(h^N_i(\tau, \cdot) - \uh^N_i (\tau, \cdot) ) \big|^2 \,  \dd \nu(z)  \ast W_N  \Big \rangle  \,  \dd \tau \, \Big|   \mathcal{F}_0  \Big]\\
 & = \frac{ 1 }{N}    \mathbb{E}  \Big[ \int_0^{\tilde{T} \wedge t^N} 
 \Big\langle h^N_i(\tau, \cdot), \int_{\R^d \setminus \zero} \big| D_{z}(h^N_i(\tau, \cdot) - \uh^N_i (\tau, \cdot) )  \big|^2 \,  \dd \nu(z)  \Big \rangle  \, \dd \tau  \, \Big|   \mathcal{F}_0  \Big].
\end{aligned}
\end{equation*}
The additional observation that 
\begin{align*}
\|h_i^N\|_{L^\infty(0,\tilde{T} \wedge t^N; L^\infty(\R^d))} \leq \frac{N_i}{N}  \kappa_N^d,
\end{align*}
the definition of the fractional Sobolev seminorm (see the Appendix), and  \eqref{equivalence} yield 
\begin{align*}
\begin{split}
 & \mathbb{E}  \Big[ \sum_{i=1}^n  \sup_{0\leq t \leq  \tilde{T} \wedge t^N }   \big|M^N_i (t)\big|\,   \Big| \mathcal{F}_0  \Big]^2 
  \lesssim    \frac{\kappa_N^d  }{N}   \sum_{i=1}^n  \mathbb{E}  \Big[  \frac{N_i}{N}   \int_0^{\tilde{T} \wedge t^N}  \hspace{-0.25 cm } \big \|  (- \Delta)^{\frac{\alpha}{2}} (h^N_i(\tau, \cdot) - \uh^N_i (\tau, \cdot))  \big\|_2^2   \,  \dd \tau \,  \Big|  \mathcal{F}_0  \Big].
 \end{split}
\end{align*}

\smallskip

\noindent \textbf{Step 4: Conclusion.} We now assume that there exists $n_1 \in \mathbb{N}$ such that 
\begin{align}
\label{assume}
\mathbb{P} \Big[\sum_{i=1}^n\frac{N_i}{N} \geq n_1 \Big] =0.
\end{align}
Then, taking the conditional expectation in \eqref{estim_sup_norm}, setting $\varepsilon= (2\alpha -1)/ (4 - 2\alpha)$, and in the martingale term using $a \leq a^2 \kappa_N^2 + \kappa_{N}^{-2}$ for $a \geq 0$, we obtain
\begin{equation} \label{final_equation}
\begin{aligned}
& \mathbb{E} \Big[  \sup_{0\leq t\leq \tilde T \wedge t^N}    \|h^N(t,\cdot)-\uh^N(t,\cdot)\|_2^2  + \int_0^{\tilde T \wedge t^N} \hspace{-0.2 cm} \big\| (- \Delta )^{\frac{\alpha}{2}} (h^N - \uh^N)( \tau, \cdot) \big\|_2^2 \, \textrm{d} \tau \, \Big | \mathcal{F}_0  \Big]  \\
& \; \;  \lesssim \|h^N(0,\cdot)-\uh^N(0,\cdot)\|_2^2 
 + \int_0^{\tilde T} \mathbb{E}  \Big[ \sup_{0\leq \xi \leq \tau \wedge t^N}  \hspace{-0.1 cm } \| h^N (\xi, \cdot) - \uh^N(\xi, \cdot) \|_2^2  \dd \tau  \, \Big | \mathcal{F}_0 \Big]    \\
&  \;   + n_1^4 ( \kappa_N^{2\alpha -3} + \kappa_N^{-1} )\tilde T +  \kappa_N^{-2} 
+n_1 \frac{ \kappa_N^{d+2} }{N} \mathbb{E} \Big[ \int_0^{ \tilde T \wedge t^N}  \big\| (- \Delta)^{\frac{\alpha}{2}} (  h^N  - \uh^N)(\tau, \cdot)\big\|_2^2 \, \dd \tau \, \Big|  \mathcal{F}_0  \Big]. 
\end{aligned}
\end{equation}
Notice that in the transition from \eqref{estim_sup_norm} to \eqref{final_equation}, we have used the upper bound on $\kappa$ included in \eqref{scaling}. 
Using the assumptions on $\kappa$ given in \eqref{scaling}, we can for $N \gg 1$ absorb the last term on the right-hand side of \eqref{final_equation} into the left-hand side to obtain
\begin{align*}
\zeta(\tilde{T})=\mathbb{E} \Big[  \| h^N - \uh^N \|^2_{[0, \tilde{T} \wedge t^N]} \, \big \vert  \mathcal{F}_0\Big] \lesssim \|h^N(0,\cdot)-\uh(0,\cdot)\|_2^2  +  \kappa_N^{2\alpha -3}  + \kappa_N^{-1} + \kappa_N^{-2} +  \int_0^{\tilde T} \zeta(\tau) \, \dd \tau  , 
\end{align*}
for $0< \tilde{T} \leq T$ with $T \in (0, T_1]$, where $T_1 = 1/ n_1^4$.  Then an application of Gr\"onwall's inequality yields 
\begin{equation*}
  \mathbb{P} \Big[ \zeta(T) \geq 2 C e^{\tilde{C}T} \delta_N^{1 + \rho} \Big ]
   \leq  \mathbb{P} \Big[  \|h^N(0,\cdot)-\uh^N(0,\cdot)\|_2^2  + \kappa_N^{-1}   \geq  2  \delta_N^{1 + \rho}  \Big]
 < \sigma(N),
\end{equation*}
where $\tilde{C} = \tilde{C}(d, n, \sigma_i, a_{ij})$ and $C = C(d, n, \sigma_i, a_{ij})$ are positive constants and $\sigma(N) \rightarrow 0$ as $N\rightarrow \infty$ by \eqref{initial_condition_assumption_2} and the lower bound on $\kappa$ from \eqref{scaling}.
To finish, similar to \cite{S_2000}, we define
\begin{align*}
\widetilde \Omega := \left\{\omega \in \Omega \, \Big \vert \,\mathbb{E} \left[ \|h^N - \hat u^N \|^2_{[0,T \wedge t^N]}  \, \big \vert \, \mathcal{F}_0 \right] (\omega)  <  2 C e^{\tilde{C}T}\delta_N^{1 + \rho}  \right\} 
\end{align*}
and, by applying Markov's inequality,  then find that 
\begin{align*}
&\mathbb{P}\left[  \|h^N - \hat u^N \|_{[0,T \wedge t^N]} \geq \delta_N \right] \leq \int_{\Omega} \mathbb{P}\big[  \|h^N - \hat u^N \|_{[0, T \wedge t^N]} \geq \delta_N \, \big| \mathcal{F}_0 \big]  \textrm{d} \mathbb{P}\\
 & \quad \leq \mathbb{P}( \widetilde  \Omega^c ) + \delta_N^{-1}  \int_{\widetilde \Omega} \mathbb{E}\Big[  \|h^N - \hat u^N \|_{[0,T \wedge t^N]}  \,  \Big | \mathcal{F}_0 \Big]  \textrm{d} \mathbb{P} \leq \sigma(N) +  2 C e^{ \tilde{C} T} \delta_N^{\rho}  \rightarrow 0 \; \textrm{ for } \; N \rightarrow \infty. 
\end{align*}
This completes our argument thanks to \eqref{prob_fact}. We can then repeat our arguments on the intervals $[T_1, 2T_1]$, $[2 T_1, 3 T_1]$, and so on, in order to obtain the result for any $T >0$. Now we can replace \eqref{assume} by \eqref{initial_condition_assumption_1}.
\end{proof}

\section{Argument for Theorem~\ref{Convergence_2}}
\label{ref_4}
Recall that $\psi$ is the function from the assumptions \eqref{assum_init_2} on the initial data in Theorem \ref{Convergence_2}. Throughout our proof of Theorem~\ref{Convergence_2}, we make use  of the following elementary relations for $\psi$.

\begin{lemma} \label{psi_new} Let $\psi(x) = \log(2 + |x|^2)$ and $\alpha \in (1/2,1)$. For all $ x\in \R^d$, the following relations hold:
$$
\begin{aligned}
&|(-\Delta)^{\alpha} \psi(x)| \lesssim_{\alpha} \psi (x),  \quad  && |\nabla^2 \psi(x)| \lesssim \psi(x),  \quad && |\nabla \psi(x)| \lesssim \psi(x),\\
&  |(-\Delta)^{\alpha} \psi^2(x)| \lesssim_{\alpha} \psi^2 (x), \quad &&  |\nabla \psi^2(x)| \lesssim \psi^2(x), \quad  &&  |\nabla^2 \psi^2(x)| \lesssim \psi^2(x).
\end{aligned}
$$
\end{lemma}

\begin{proof} The second,  third, fifth and sixth relations follow from  simple computations. 
For the first relation we split the integral in the definition of the fractional Laplacian into two contributions:
\begin{align}
\label{near_far_field}
\begin{split}
 (-\Delta)^{\alpha} \psi(x) & =\lim_{\epsilon \rightarrow0}  \int_{B_{1}(x) \setminus B_{\epsilon}(x)} \frac{\psi(x) - \psi(y)}{|x - y|^{d + 2\alpha}} \dd y  +  \int_{\R^d \setminus B_{1}(x)} \frac{\psi(x) - \psi(y)}{|x - y|^{d + 2\alpha}} \dd y.
\end{split}
\end{align}
Then, for the first term on the right-hand side we write 
\begin{align*}
\Big| \int_{B_{1}(x) \setminus B_{\epsilon}(x)} \frac{\psi(x) - \psi(y)}{|x - y|^{d + 2\alpha}} \dd y \Big|& 
 \leq  \int_{B_{1}(x)  \setminus B_{\epsilon}(x)} \frac{|\psi(x) - \psi(y)- \nabla \psi(x) \cdot (x-y)|}{|x - y|^{d + 2\alpha}} \dd y\\
&  \leq  \int_{B_{1}(x)} \frac{\|\nabla^2 \psi\|_{\infty}}{|x - y|^{d + 2\alpha -2}} \dd y \lesssim 1.
\end{align*}
For the second term of \eqref{near_far_field},
using $\psi(y) \lesssim \psi(x) + \psi(x-y)$ and $\psi(x)\lesssim |x|^{\alpha /2}$ for $|x| \geq 1$, we write
\begin{align}
\label{correct_arxiv_1}
\begin{split}
 \int_{\R^d \setminus B_{1}(x) } \frac{\psi(y)}{|x - y|^{d + 2\alpha}} \dd y
 \lesssim \psi(x) + \int_{\R^d \setminus B_1(x) } \frac{|x-y|^{\alpha/2}}{|x-y|^{d + 2\alpha}} \dd y \lesssim \psi(x).
\end{split}
\end{align}
Notice that the relation $\psi(y) \lesssim \psi(x) + \psi(x-y)$ follows from the observation that 
\begin{align}
 \label{calculation}
\psi(2 x) = \log(2 + |2 x|^2) \leq \log(4) + \log(2+ |x|^2) \lesssim \psi(x).
\end{align}
In particular, if $|x| \geq |y|/2$, then $\psi(y) \lesssim \psi(2x) + \psi(x-y)$ and \eqref{calculation} can be applied. Likewise, if $|x| \leq |y|/2$, then $2 |y - x| > 2| |y| - |x| | \geq |y|$ and this gives $\psi(y) \leq \psi(x) + \psi(2 (y - x ))$.

For the fourth relation, we use exactly the same argument as for the first.
\end{proof}

\subsection{Proof of Theorem \ref{Convergence_2}}  Our proof follows the arguments in \cite[Theorem 2]{O_89} and \cite[Theorem 6.3]{S_2000} with adaptions made to take into account the L\'{e}vy noise and fractional cross-diffusion. We follow quite closely the proof of Theorem 6.3 in \cite{S_2000},
however, since our setting requires various simple modifications, we include the full argument for completeness. 

\begin{proof}[Proof of Theorem \ref{Convergence_2}]
Let $f \in \mathcal{B}_1$, where 
$$
\mathcal{B}_1 := \left\{ f \in C^1_b(\R^d)  \, \;  | \, \;  \| f\|_{\infty} + \| \nabla f \|_{\infty} \leq 1 \right\}.
$$
We decompose $f =  f_R + \hat{f}_R$, where $\textrm{supp}(f_R) \subseteq B_R$ and $\textrm{supp}(\hat{f}_R) \subseteq \R^d \setminus B_{R-2}$, for $R>2$.  For any $t> 0$ and $\psi(x) = \log(2 + |x|^2)$, we obtain
\begin{equation} \label{thm6.3_1} 
 \begin{aligned}
 \left| \langle S_i^N(t)- u_i(t, \cdot), f \rangle \right| & \lesssim R^{\frac d 2}   \big( \| h_i^N(t)- \hat{u}_i^N(t, \cdot) \|_{2}+   \| \hat{u}_i^N(t, \cdot) - u_i(t, \cdot)\|_{2} \big) \\
  & \quad +  \frac{1}{\psi(R)} \langle S_i^N(t) + u_i(t, \cdot), \psi \rangle + \kappa_N^{-1} \langle S_i^N(t) , 1 \rangle.
    \end{aligned}
\end{equation}
For the details of this estimate we point the reader to (61) in \cite[Theorem 6.3]{S_2000}. Here we require the positivity of $u_i$,  which is shown in Theorem~\ref{existence_regular}. By \eqref{thm6.3_1},  using the stopping time $t^N$ defined in \eqref{stopping_time} and the convergence results for $\| h^N(t, \cdot)- \hat{u}^N(t, \cdot) \|_2$ and $\| \hat{u}^N(t, \cdot) - u(t, \cdot)\|_2$, shown in  Theorems~\ref{Convergence}~and~\ref{full_system} respectively,  it suffices to show
\begin{align}
\label{conclusion_thm_3}
\lim_{R\rightarrow \infty} \lim_{N \rightarrow \infty } \mathbb{P} \Big[   \sum_{i=1}^n \sup_{0\leq t \leq T}  \langle S_i^N(t \wedge t^N)  + u_i(t \wedge t^N, \cdot), \psi  \rangle  \psi^{-1}(R) \geq  \mu  \Big] = 0,
\end{align}
for any $\mu>0$. To obtain \eqref{conclusion_thm_3} we consider equation~\eqref{ito} for $t\wedge t^N$.  Using that $|\nabla \psi| \lesssim \psi$  and  $| (-\Delta)^{\alpha}\psi| \lesssim \psi$, see Lemma~\ref{psi_new}, together with the regularity of $\hat{s}^N_j$, we  obtain  
\begin{align}
\label{ito_lemma_3_2}
\begin{split}
 \langle  S_i^N(t \wedge t^N) , \psi  \rangle  \lesssim  \langle  S_i^N(0) , \psi  \rangle +  \int_0^{t}  \langle  S_i^N(\tau  \wedge t^N), \psi \rangle \dd \tau +  \big| M_i^{N,1}(t \wedge t^N) \big|,
\end{split}
\end{align}
where 
\begin{align*}
M^{N,1}_i(t) :=  \frac{1}{N} \sum_{k=1}^{N_i} \int_{0}^{t \wedge t^N} \int_{\R^d \setminus \left\{ 0  \right\} } D_z  \psi (X^{k,N}_i(\tau_-) )  \tilde{\mathcal{N}}_i^k( \dd z \dd \tau ).
\end{align*}
An application of Gr\"onwall's inequality to \eqref{ito_lemma_3_2} gives that 
\begin{align}
\label{ito_lemma_3_2_final}
\sup_{0\leq t \leq T}  \langle S_i^N(t \wedge t^N) , \psi  \rangle  \lesssim_T \langle  S_i^N(0) , \psi  \rangle +  \sup_{0\leq t \leq T} \big| M^{N,1}_i(t\wedge t^N)\big|;
\end{align}
the analogue of \eqref{ito_lemma_3_2_final} in \cite{S_2000} is the estimate following (65).

We estimate the martingale $M^{N,1}_i(t)$ using similar methods as in the proof of Theorem \ref{Convergence}. In particular, we use the independence of the L\'{e}vy processes and apply the optional sampling theorem, Doob's $L^p$-inequality, and the It\^{o} isometry to write
\begin{equation}\label{martingale_M_2}
\begin{aligned}
&\mathbb{E} \Big[ \sup_{0\leq t \leq T} |M^{N,1}_i (t \wedge t^N)| \Big| \mathcal{F}_0\Big]^2 
 \lesssim  \frac{1}{N}\mathbb{E}  \Big [ \int_0^{\tilde{T}\wedge t^N}  \hspace{-0.2 cm }  \Big\langle S^N_i(\tau, \cdot), \int_{\R^d \setminus \zero} | D_z \psi |^2  \dd \nu(z) \Big\rangle   \dd  \tau \big |  \mathcal{F}_0  \Big].
\end{aligned}
\end{equation}
To continue we emulate the argument from Lemma \ref{psi_new} and obtain
\begin{align*}
\begin{split}
\int_{\R^d \setminus \zero} \big| D_z \psi (x)|^2  \dd \nu(z) & \lesssim  \int_{B_1(0)} \frac{ \| \nabla  \psi \|_{\infty}^2 }{|z|^{d+2\alpha -2}} \dd z + \int_{\R^d \setminus B_1(0)} \frac{ \psi^2 (x) +  \psi^2 (z + x)}{|z|^{d+2\alpha}} \dd z\\
& \lesssim 1+ \psi^2(x)+  \int_{\R^d \setminus B_1(0)} \frac{\psi^2 (z ) + \psi^2(x)}{|z|^{d+2\alpha}} \dd z \lesssim  \psi^2(x),
\end{split}
\end{align*}
where we have used that $\psi^2(x) \lesssim |x|^{\alpha}$ for $|x|\geq 1$. This estimate is then combined with \eqref{martingale_M_2}.
To handle the resulting right-hand side, we again use It\^{o}'s formula, now with $\psi^2$,  in conjunction with $|\nabla \psi^2| \lesssim \psi^2$ and $|(- \Delta)^{\alpha} \psi^2| \lesssim \psi^2$ from Lemma \ref{psi_new}. We find that 
\begin{align}
\label{ito_lemma_3_2_1}
\begin{split}
\langle  S_i^N(t \wedge t^N) , \psi^2  \rangle \lesssim  \langle  S_i^N(0) , \psi^2 \rangle + \int_0^{t}  \langle  S_i^N(\tau  \wedge t^N), \psi^2 \rangle \dd \tau +  M^{N,2}_i(t \wedge t^N),
\end{split}
\end{align}
for $i=1, \ldots, n$, where $M^{N,2}_i$ are martingales with $M^{N,2}_i(0) = 0$. Taking the conditional expectation of \eqref{ito_lemma_3_2_1} and applying Gr\"{o}nwall's inequality yields
\begin{align}
\label{ito_lemma_3_3}
\begin{split}
\sup_{0\leq t \leq T} \mathbb{E} \big[ \langle  S_i^N(t \wedge t^N) , \psi^2  \rangle \big | \mathcal{F}_0   \big] \lesssim_T \langle  S_i^N(0) , \psi^2 \rangle;
\end{split}
\end{align}
the analogue of this estimate in \cite{S_2000} is (63). After an application of the Fubini theorem this allows us to bound the right-hand side of \eqref{martingale_M_2} by $\langle  S_i^N(0) , \psi^2 \rangle$, up to a multiplicative constant depending on $T$. 

To finish, we now take the conditional expectation of \eqref{ito_lemma_3_2_final} to obtain
\begin{align}
\label{final_thm3_1}
\mathbb{E}\big[   \sup_{0\leq t \leq T} \langle  S_i^N(t \wedge t^N) , \psi \rangle \big | \mathcal{F}_0 \big] \lesssim_T  \langle S_i^N (0), \psi^2  \rangle +1.
\end{align}
Similar estimates, now using the weak formulation of \eqref{sys.final} instead of the It\^{o} formula, ensure 
\begin{align}
\label{final_thm3_2}
\sup_{0\leq t \leq T} \langle  u_i(t, \cdot), \psi \rangle \lesssim_T \langle u_{i}^0, \psi \rangle,
\end{align}
where we know that the right-hand side is finite due to our assumption on $u_{i}^0$ in  \eqref{assum_init_2}. 

To conclude the proof of Theorem~\ref{Convergence_2}, we combine \eqref{conclusion_thm_3}, \eqref{final_thm3_1}, and \eqref{final_thm3_2}, together with the assumptions on the initial condition given in  \eqref{assum_init_2}. 
\end{proof}

\section{Proof of Theorem~\ref{existence_regular}} 
\label{ref_5}
 

\begin{definition} 
A weak solution of \eqref{syst:reg} is $\uh^N \in L^2(0,T; H^\alpha(\R^d))^n \cap L^{\infty}(0,T;L^2(\R^d))^n$ with $\partial_t \uh^N \in L^2(0,T; H^\alpha(\R^d)^\prime)^n$ that satisfies the system \eqref{syst:reg}  in the variational form 
\begin{equation} \label{def_sol:sys.reg}
\begin{aligned}
&\int_0^T \big\langle \partial_t \uh_i^N,  \psi_i \big \rangle_{(H^{\alpha})^\prime, H^\alpha } \, \dd t  +\int_0^T \sigma_i\big\langle  (-\Delta)^{\frac\alpha 2}   \uh_i^N, (-\Delta)^{\frac\alpha 2}  \psi_i \big\rangle \, \dd t\\
& +
\sum_{j=1}^n \int_0^T a_{ij}\big\langle   (-\Delta)^{\frac{1-\alpha} 2} \big( \hat u_i^N \nabla^\beta (\uh_j^N \ast \hat W_N)\big), \nabla (-\Delta)^{\frac{\alpha-1} 2} \psi_i \big\rangle \,  \dd t= 0,  
\end{aligned} 
\end{equation} 
for  $\psi_i \in L^2(0,T; H^\alpha(\R^d))$, where $i=1, \ldots, n$. The initial condition is satisfied in the $L^2$-sense. 
\end{definition} 

\noindent Weak solutions of \eqref{sys.final} are defined in the analogous way. Here,  $\langle  \phi, \psi  \rangle_{ (H^{\alpha})^\prime, H^\alpha } $ denotes the dual pairing between $\phi \in L^2(0,T; H^\alpha(\R^d)^\prime)$ and $\psi \in L^2(0,T; H^\alpha(\R^d))$.

\begin{proof} [Proof of Theorem~\ref{existence_regular}] This proof proceeds in five steps. In the first step, we use a Galerkin argument to prove the  existence of a weak solution for a linearization of the regularized system~\eqref{syst:reg}.  
In the second step, we transition from the linearized problem to the system~\eqref{syst:reg} using a Banach fixed-point argument. In Steps 3 and 4, we prove~\eqref{estim_Hs_1} and~\eqref{estim_hu_N} for local solutions of  \eqref{syst:reg}. In Step 5 we show that for small enough initial data, we can construct a global solution that also satisfies the estimates \eqref{estim_Hs_1} and \eqref{estim_hu_N}. 
\smallskip 

\noindent \textbf{Step 1: Existence of  a local weak solution for a linearization of \eqref{syst:reg}.} We first consider the following linearized version of \eqref{syst:reg}  
\begin{equation} \label{lin.sys.reg}
\begin{aligned}
&\partial_t \uh_i^N + \sigma_i (-\Delta)^\alpha  \uh_i^N - \diver\Big(\sum_{j=1}^n a_{ij} v_i^N \nabla^\beta (\uh_j^N \ast \hat W_N) \Big)= 0  \; && \; \text{ in } \;  (0,T)\times \R^d,  \\
& \uh_i^N(0) = u^0_{i} &&\;   \text{ in } \; \R^d,
\end{aligned} 
\end{equation} 
for a given $v^N  \in L^2(0, T;  H^{\alpha} (\R^d))^n\cap L^\infty(0, T; L^2(\R^d))^n$ and $i=1, \ldots, n$. To show existence of a solution of  \eqref{lin.sys.reg} we take a Galerkin approximation $\{\hat u^{N,k}\}_{k \in \mathbb N}$ with
\begin{equation}\label{Galerkin}
\uh_i^{N,k}(t,x)= \sum_{l=1}^k \rho_{il}^{N,k}(t) q_l(x), 
\end{equation}
where the span of the elements $\{ q_l\}_{l\in \mathbb N} $ is dense in $H^\alpha(\R^d)$ and they are pairwise orthonormal in $L^2(\R^d)$,  satisfying 
\begin{equation} \label{def_sol:Galerkin}
\begin{aligned}
& \int_{\R^d} \Big[ \partial_t \uh_i^{N,k} q_l + \sigma_i  (-\Delta)^{\frac\alpha 2}   \uh_i^{N,k}  (-\Delta)^{\frac\alpha 2}  q_l \Big] \, \dd x   \\
&+ \int_{\R^d} 
 \sum_{j=1}^n  a_{ij}  (-\Delta)^{\frac{1-\alpha} 2} \big(  v_i^N \nabla^\beta (\uh_j^{N,k} \ast \hat W_N)\big) \nabla (-\Delta)^{\frac{\alpha-1} 2} q_l \, \dd x = 0, \quad \text{for } l \in \mathbb{N}.
\end{aligned} 
\end{equation} 
We remark that by \eqref{frac_Leib_2}, since $ \nabla^{\beta} \uh_j^{N,k} (t)\ast \hat W_N \in W^{1,\infty}(\R^d)$ and $v_i^N \in L^2(0,T; H^{\alpha}(\R^d))$,  the expression $(-\Delta)^{\frac{1-\alpha} 2} (  v_i^N (t)\nabla^\beta (\uh_j^{N,k}(t) \ast \hat W_N)) \in L^2(\R^d)$ is well-defined. Now, by standard ODE theory, there exist unique  $\rho_{il}^{N,k} \in H^1(0,T)$ such that $\uh_i^{N,k}$, defined by \eqref{Galerkin}, are solutions of \eqref{def_sol:Galerkin} with $\uh_i^{N,k}(0) = u_{i}^{0,k}$, where $u_{i}^{0,k}$ are the projections of $u^0_{i}$ onto ${\rm Span}\{q_1, \ldots, q_k\}$.  


We now derive \textit{a priori} estimates that are uniform in $k \in \mathbb{N}$. Considering $\uh_i^{N,k}$ as a test function in  \eqref{def_sol:Galerkin}, integrating with respect to the time variable, summing over $i= 1, \dots, n$, and using Young's inequality we obtain 
\begin{align}\label{energy_1}
\begin{aligned}
& \sum_{i=1}^n \int_0^\tau  \frac{d}{dt} \int_{\R^d}  |\uh_i^{N,k}|^2\dd x \dd t + \sum_{i=1}^n 2 \sigma_i \int_0^\tau \int_{\R^d}  |(-\Delta)^{\frac \alpha 2} \uh_i^{N,k}|^2 \dd x  \dd t  \\
& \; \leq   \sum_{i,j=1}^n \int_0^\tau \int_{\R^d}  \Big[C_\varsigma \big|(-\Delta)^{\frac{1-\alpha}2} \big(v^N_i\big(\nabla^\beta\uh_j^{N,k} * \hat W_N\big)\big)\big|^2   +  \varsigma   \big|(-\Delta)^{\frac \alpha 2} \uh_i^{N,k}\big|^2 \Big] \dd x  \dd t,
\end{aligned}
\end{align}
for any $\tau \in (0, T]$. Notice that here we have used  equivalence \eqref{equivalence} from the appendix. Using~\eqref{frac_Leib_2} and the Gagliardo-Nirenberg interpolation inequality,  we obtain 
\begin{equation*}
\begin{aligned}
 &\big\| (-\Delta)^{\frac{1-\alpha}2} \big(v_i^N (t, \cdot)  \nabla^\beta\uh_j^{N,k} (t, \cdot) * \hat W_N \big) \big\|_{2}   \lesssim  \| v_i^N(t, \cdot)\|_{H^{1-\alpha }}
 \|  \uh_j^{N,k}(t, \cdot) \ast \nabla^{\beta} \hat W_N \|_{ W^{1,\infty}}  
 \\
 & \lesssim \| \hat W_N \|_{H^{1+\beta}}  \| v_i^N(t, \cdot)\|_{ H^{1-\alpha }} \|  \uh_j^{N,k}(t, \cdot) \|_{2} \lesssim_N  \| v_i^N(t, \cdot)\|^{\frac{1-\alpha}{\alpha}}_{H^{\alpha}} \| v_i^N(t, \cdot)\|_{2}^{\frac{2 \alpha -1} {\alpha}} \|  \uh_j^{N,k}(t, \cdot) \|_{2},  
 \end{aligned}
\end{equation*}
for $t \in (0,\tau]$. 
 Combining the previous estimate with \eqref{energy_1} and using Gr\"onwall's lemma along with an application of H\"{o}lder's inequality (in the time integral), we obtain 
\begin{equation} \label{Gron_apply_1}
\begin{aligned}
 \sup_{t \in(0, T]}\|\uh^{N,k}(t, \cdot)\|_{2}^2 & \lesssim_{N} \| u_0\|^2_2 \exp\big( \|v^N \|^2_{L^{\infty}(0,T; L^2(\R^d))} \|v^N \|^{\frac{2(1 - \alpha)}{ 2\alpha - 1}}_{L^{2}(0,T;H^{\alpha}(\R^d))} C(N) T \big),   \\
  \| \uh^{N,k} \|^{2}_{L^2(0,T; H^{\alpha}(\R^d))} & \lesssim_{N}  \| u_0\|^2_2 \Big( 1+  \, T \, \|v^N \|^2_{L^{\infty}(0,T; L^2(\R^d))} \|v^N \|^{\frac{2(1 - \alpha)}{ 2\alpha - 1}}_{L^{2}(0,T;H^{\alpha}(\R^d)) }  \Big)  \\
 & \qquad  \times   \exp\big( \|v^N \|^2_{L^{\infty}(0,T; L^2(\R^d))} \|v^N \|^{\frac{2(1 - \alpha)}{ 2\alpha - 1}}_{L^{2}(0,T;H^{\alpha}(\R^d)) } C(N) T \big).
 \end{aligned}
\end{equation}
By \eqref{Gron_apply_1}, it follows directly from  \eqref{def_sol:Galerkin} that
$$
\| \partial_t  \uh^{N,k}\|_{L^2(0, T; H^{\alpha}(\R^d)^\prime)} \leq C(\| u_0\|_2, \|v^N \|_{L^{\infty}(0,T; L^2(\R^d))} ,\|v^N \|_{L^{2}(0,T;H^{\alpha}(\R^d))}, N).
$$
Since the constants above are independent of $k$, we pass to a weakly convergent subsequence 
\begin{equation} \label{convergences}
 \uh^{N,k}   \rightharpoonup^\ast  \hat{u}^N   \;  \text{ in }  L^\infty(0,T; L^2(\R^d))^n \; \text{ and } \;   \uh^{N,k}   \rightharpoonup  \hat{u}^N   \, \text{ in }  L^2(0,T; H^\alpha(\R^d))^n \text{ as } k \to \infty. 
 \end{equation} 
Integrating \eqref{def_sol:Galerkin} in time and passing $k \rightarrow \infty$ yields  $\hat{u}^N \in L^\infty(0,T; L^2(\R^d))^n \cap L^2(0,T; H^{\alpha}(\R^d))^n$ as a weak  solution of \eqref{lin.sys.reg} with $\partial_t \hat{u}^N \in L^2(0,T, H^{\alpha}(\R^d)^{\prime})^n$. In order to pass to the limit in the third term of \eqref{def_sol:Galerkin}, we write 
\begin{align*}
& \int_{\R^d}  (-\Delta)^{\frac{1-\alpha} 2} \big(  v_i^N \nabla^\beta (\uh_j^{N,k} \ast \hat W_N)\big) \nabla (-\Delta)^{\frac{\alpha-1} 2} \psi_i  \, \dd x  =  \int_{\R^d}    v_i^N  (\uh_j^{N,k} \ast \nabla^\beta \hat W_N)   \nabla  \psi_i  \, \dd x.
\end{align*}
Then notice that $\nabla^{\beta} ( \uh_j^{N,k} \ast \hat W_N)  \rightharpoonup \nabla^{\beta} ( \uh_j^{N} \ast \hat W_N)$ weakly in $L^2(0,T; L^2(\R^d))$ and consider $\psi \in C^{\infty}_0(0, T; C^{\infty}_0 (\R^d))$. A standard argument shows that the initial condition is satisfied in the $L^2$-sense. 

We remark that by the lower semicontinuity of the norms, we obtain \eqref{Gron_apply_1}  also for  $\uh^N$. Standard arguments yield the uniqueness of solutions of problem~\eqref{lin.sys.reg}.

\smallskip 

\noindent \textbf{Step 2: Existence of local solutions for \eqref{syst:reg}.}  
 To show existence of a local solution of the nonlinear problem \eqref{syst:reg} we apply the Banach fixed point theorem in the space 
\begin{align*}
& \mathcal{X}: =\Big\{ v \in  L^2(0,T; H^{\alpha}(\R^d))^n\cap L^{\infty}(0,T;  L^2(\R^d))^n \, : \, \\
& \qquad \qquad \qquad \qquad \qquad  \| v \|^2_{ L^2(0,T; H^{\alpha}(\R^d))} +  \| v \|^2_{L^{\infty}(0,T; L^2(\R^d))} \leq 3 C^{\prime}(N) \|u^0\|^2_{2} \Big\},
\end{align*} 
where $C^{\prime}(N)$ is the maximum of the universal constants appearing in \eqref{Gron_apply_1}. In particular, we consider the following mapping
$$
 \mathcal K: \mathcal X \to \mathcal X,  \qquad v^{N}\stackrel{\mathcal K}{\longmapsto} \hat{u}^N,
$$
where $\uh^N$ is the unique weak solution of the linear problem \eqref{lin.sys.reg} provided by the previous step. Notice that by \eqref{Gron_apply_1}, for $T:= T(\|u^0\|_2, N)$ small enough  this mapping is a self-map of~$\mathcal X$.

We now show that for $T:=T(\|u^0\|_2,N)>0$ small enough, the mapping $\mathcal K$ is a contraction on $\mathcal X$. For this, we let $v_1^N \mapsto \hat{u}_1^N$ and $ v_1^N \mapsto \hat{u}_2^N$ and see that $\hat{u}_1^N - \hat{u}_2^N$ satisfies
\begin{align*}
\begin{split}
&\sup_{t \in (0,T]} \| \hat{u}_{1,i}^N - \hat{u}_{2,i}^N \|^2_{2} + \int_0^T \big\| (-\Delta)^{\frac{\alpha}{2}} (\hat{u}_{1,i}^N - \hat{u}_{2,i}^N) \big\|^2_{2} \, \dd t \\
& \lesssim \int_0^T \int_{\R^d}  \Big| \nabla (- \Delta)^{\frac{\alpha-1}{2}} ( \hat{u}_{1,i}^N - \hat{u}_{2,i}^N ) 
\sum_{j=1}^n (-\Delta)^{\frac{1-\alpha}{2}} \big(v^N_{1,i} \nabla^{\beta} (\hat{u}_{1,j}^N - \hat{u}_{2,j}^N)  \ast \hat{W}_N\big)  \Big| \, \dd x \, \dd t \\
& \qquad + \int_0^T \int_{\R^d} \Big| \nabla (- \Delta)^{\frac{\alpha-1}{2}} ( \hat{u}_{1,i}^N - \hat{u}_{2,i}^N ) 
 \sum_{j=1}^n  (-\Delta)^{\frac{1-\alpha}{2}} \big((v^N_{1,i}  - v^N_{2,i}) \nabla^{\beta} \hat{u}_{2,j}^N \ast \hat{W}_N \big)\, \Big|  \dd x \, \dd t\\
& \lesssim  \int_0^T \Big[ \varsigma  \big\|  (- \Delta)^{\frac{\alpha}{2}}(\hat{u}_{1,i}^N - \hat{u}_{2,i}^N)  \big\|^2_{2} +\sum_{j=1}^n C_{\varsigma}\Big( \|v_{1,i}^N \|^2_{H^{1-\alpha}} \big\| \nabla^{\beta} ( \hat{u}_{1,j}^N - \hat{u}_{2,j}^N) \ast \hat{W}_N \big\|^2_{W^{1,\infty}} \\
& \hspace{3cm} +  \| v_{1,i}^N - v_{2,i}^N\|^2_{ H^{1-\alpha}} \big\| \nabla^{\beta} \uh^N_{2,j} \ast \hat{W}_N \big\|^2_{W^{1,\infty} }\Big)  \Big] \, \dd t,
\end{split}
\end{align*}
for $\varsigma>0$. Here we have used the relation \eqref{frac_Leib_2}. We use \eqref{GN_form} and Young's inequality for convolutions, to continue the above estimate as 
\begin{align*}
&\sup_{t \in (0,T]}  \| \hat{u}_{1,i}^N - \hat{u}_{2,i}^N \|^2_{2} +  \int_0^T \big\| (-\Delta)^{\frac{\alpha}{2}} (\hat{u}_{1,i}^N - \hat{u}_{2,i}^N) \big\|^2_{2} \, \dd t \\
&  \lesssim_N  \int_0^T \hspace{-0.2 cm } \big[C_{\varsigma^{\prime}} \|v_{1,i}^N \|^2_{2}  +  \varsigma^{\prime}  \|v_{1,i}^N \|^2_{ H^{\alpha}}\big]  \|\hat{u}_{1}^N - \hat{u}_{2}^N\|^2_{2} +  \big[C_{\varsigma^{\prime \prime}} \|v_{1,i}^N - v_{2,i}^N \|^2_{2}  +  \varsigma^{\prime \prime}  \|v_{1,i}^N - v_{2,i}^N\|^2_{ H^{\alpha}}\big]   \| \uh^N_2\|^2_{2}  \, \dd t,
\end{align*}
for $\varsigma^{\prime}$ and $\varsigma^{\prime \prime}>0$. Treating the terms on the right-hand side in more detail, we obtain
\begin{align*}
& \int_0^T\big[C_{\varsigma^{\prime}} \|v_{1,i}^N \|^2_{2}  +  \varsigma^{\prime}  \|v_{1,i}^N \|^2_{ H^{\alpha}} \big]  \|\hat{u}_{1}^N - \hat{u}_{2}^N\|^2_{2} \dd t\\
& \quad  \leq C_{\varsigma^{\prime}} \sup_{t\in (0,T]} \|v_{1,i}^N \|^2_{2} \sup_{t\in (0,T]} \|\hat{u}_{1}^N - \hat{u}_{2}^N\|^2_{2} T + \varsigma^{\prime}  \|v_{1,i}^N \|^2_{L^2(0,T ; H^{\alpha}(\R^d))}  \sup_{t \in (0,T]} \| \hat{u}_{1}^N - \hat{u}_{2}^N \|^2_{2}\\
& \quad  \leq 3 C^{\prime}(N) \|u^0\|_2^2 \big( C_{\varsigma^{\prime}} T \sup_{t\in (0,T]} \|\hat{u}_{1}^N - \hat{u}_{2}^N\|^2_{2}  + \varsigma^{\prime}   \sup_{t \in (0,T]} \| \hat{u}_{1}^N - \hat{u}_{2}^N \|^2_{2} \big)
\end{align*}
and, in exactly the same way, we find that 
\begin{align*}
& \int_0^T(C_{\varsigma^{\prime \prime}} \|v_{1,i}^N - v_{2,i}^N \|^2_{2}  +  \varsigma^{\prime \prime}  \|v_{1,i}^N - v_{2,i}^N\|^2_{H^{\alpha}})   \| \uh^N_2\|^2_{2} \dd t\\
& \quad \leq 3 C^{\prime}(N) \|u^0\|_2^2 \big( C_{\varsigma^{\prime \prime}} T \sup_{t\in (0,T]} \|v_{1,i}^N - v_{2,i}^N\|^2_{2}  + \varsigma^{\prime \prime}  \| v_{1,i}^N - v_{2,i}^N \|^2_{L^2(0,T ; H^{\alpha}(\R^d))} \big).
\end{align*}
Summing over $i = 1,\dots, n$ and choosing appropriate  $T$,   depending on $\| u^0\|_{2}^2$ and $N$,  and  $\varsigma^{\prime}>0$, we obtain 
\begin{equation*}
\begin{aligned}
&  \sup_{t\in(0,T]} \| \hat{u}_{1}^N - \hat{u}_{2}^N \|^2_{2} +  \int_0^T \big\| (-\Delta)^{\frac{\alpha}{2}} (\hat{u}_{1}^N - \hat{u}_{2}^N) \big\|^2_{2} \, \dd t \\
 & \qquad \qquad   \leq  3 C^{\prime}(N) \|u^0\|_2^2 \big( C_{\varsigma^{\prime \prime}} T \sup_{t\in (0,T]} \|v_{1}^N - v_{2}^N\|^2_{2}  + \varsigma^{\prime \prime}  \| v_{1}^N - v_{2}^N \|^2_{L^2(0,T ; H^{\alpha}(\R^d))} \big).
\end{aligned}
\end{equation*}
Possibly choosing a smaller $\varsigma^{\prime \prime}$ and $T$, this shows that for $T:= T(\|u^0\|_2^2, N)$ small enough  the mapping $\mathcal K$ is a contraction on $\mathcal X$. 

By the Banach fixed-point theorem we obtain a unique fixed point of the mapping $\mathcal K$ in the set $\mathcal X$. This fixed point is a local solution of \eqref{syst:reg} up to the time $T:=T(\|u^0\|_2^2, N)$.

\smallskip

\noindent \textbf{Step 3: Higher-order \textit{a priori} estimates for solutions of \eqref{syst:reg}.}
In this step we show that $\uh^{N}\in L^2(0,T; H^{s+\alpha}(\R^d))^n$, where $u^{0} \in H^{s}(\R^d)^n$. The distinction between the current step and the next is that here we allow the constants in our estimates to depend on $N$.

Let $ \tau \in (0, T]$, where this is the interval of existence of the local solution $\uh^N$. Taking $\psi_i = D^l_{-h} D^l_h \hat u^N_i$, for $l =1, \dots, s$, as a test function in \eqref{def_sol:sys.reg} and using estimate  \eqref{frac_Leib_2} yields
\begin{align*}
& \big \|  D^l_h \uh^{N}_i (\tau)\big \|_{2}^2 +  \int_0^\tau \big\|  (-\Delta)^{\frac \alpha 2}  D^l_h  \uh_i^{N} \big\|^2_{2}  \,  \dd t 
\\
&\lesssim \big\|  D^l_h u^{0}_i \big \|_{2}^2 +  \sum_{j=1}^n   \int_0^\tau  \sum_{m=1}^l  \big\| D^{m}_h \uh_i^N \big\|_{H^{1-\alpha}}^2 \big\|  \uh_j^N \ast  D_h^{l-m}   \nabla^\beta   \hat  W_N \big\|_{W^{1, \infty}}^2 \,  \dd t 
\\ &\lesssim_N  \big\|  D^l_h u^{0}_i \big \|_{2}^2  +  \sum_{j=1}^n  \int_0^\tau  \sum_{m=1}^l \big \| D^{m}_h \uh_i^N \big\|_{H^{\alpha}}^{\frac{2(1-\alpha)} \alpha} \big \| D^{m}_h \uh_i^N \big\|_{2}^{\frac{2(2\alpha-1)} \alpha} \| \uh_j^N \|_2^2 \,  \dd t \\
& \lesssim_N  \big\|  D^l_h u^{0}_i \big \|_{2}^2  +   \int_0^\tau \sum_{m=1}^l \Big(   \big\| D^{m}_h \uh_i^N\|_{2}^2  \big\| \uh^N \|_2^2 +  \varsigma \big\| (-\Delta)^{\frac{\alpha}{2}} D^{m}_h \uh_i^N \big\|_{2}^{2} + C_\varsigma \big\| D^{m}_h \uh_i^N \big\|^{2}_2  \| \uh^N \|_2^{\frac{2\alpha}{2\alpha-1}} \Big) \, \dd t , 
\end{align*}
for $i=1, \ldots, n$ and  $\varsigma>0$. 
Summing over $l$ and $i$ gives 
 \begin{equation*}
\begin{aligned}
 & \sum_{l=1}^s  \| D_h^l \uh^{N}(\tau)\|^2_{2}  +   \int_0^\tau \sum_{l=1}^s  \big\| (-\Delta)^{\frac \alpha 2} D_h^l\uh^{N}\big\|^2_{2} \,  \dd t \\
& \quad \lesssim \sum_{l=1}^s \Big( \big\|  D^l_h u^{0} \big\|_{2}^2 +    \big(\|\uh^N\|^{\frac {2\alpha}{2\alpha- 1}}_{L^{\infty}(0, \tau; L^2(\R^d))}    +  \|\uh^N\|^2_{L^{\infty}(0, \tau; L^2(\R^d))} \big) \int_0^\tau \big\| D_h^{l}  \uh^{N}\big\|_{2}^{2} \, \dd t \Big).
 \end{aligned}
\end{equation*}
Thus, the regularity assumption on $u^0$ and applying the Gr\"onwall inequality yields
\begin{align*}
\sum_{m=1}^s  \Big(\big\| D_h^m  \uh^{N} \big\|_{L^\infty(0,T; L^2(\R^d))} +  \big\| D_h^m  \uh^{N}\big\|_{L^2(0,T; H^{\alpha}(\R^d))} \Big) \le C(N), 
\end{align*}
where $C(N)>0$ is independent of $h$, and hence 
\begin{align*}
 \| \uh^{N}\|_{L^\infty(0,T; H^s(\R^d))} + \| \uh^{N}\|_{L^2(0,T; H^{s+\alpha}(\R^d))}  \le C(N). 
\end{align*}

\noindent \textbf{Step 4: Uniform in $N$ higher-order estimates for solutions of \eqref{syst:reg}.} In this step we show \eqref{estim_Hs_1} and \eqref{estim_hu_N}. The main difficulty is showing that there exists $s^{\prime}<s$ such that 
\begin{equation}
\label{estim_ho_1}
 \begin{aligned}
&\sum_{i=1}^n \frac{d}{dt} \| \uh^{N}_i\|^2_{H^s}  + \sum_{i=1}^n \tilde{\sigma} \big\| (-\Delta)^{\frac \alpha 2} \uh^{N}_i \big\|^2_{H^s}  \\
 &\lesssim \sum_{i,j=1}^n \Big[ \|\hat u^N_j\|_{H^s} \big\|(-\Delta)^{\frac \alpha 2}\hat u^N_i\big\|_{H^{s'}} + 
\big\|(-\Delta)^{\frac \alpha 2}\hat u^N_j \big\|_{H^{s^\prime}}  \|\hat u^N_i\|_{H^s}  \Big]\big\|(-\Delta)^{\frac\alpha 2}\hat u^N_i\big\|_{H^s} \\ \end{aligned}
\end{equation}
holds, where $\tilde{\sigma}>0$. To see that \eqref{estim_ho_1} is sufficient for \eqref{estim_Hs_1} and \eqref{estim_hu_N}, notice that 
\begin{align*}
  \|\hat u^N_j\|_{H^s} \big\|(-\Delta)^{\frac \alpha 2}\hat u^N_i \big\|_{H^{s'}}  \lesssim \|\hat u^N_j\|_{H^s}       \big \|(-\Delta)^{\frac \alpha 2}\hat u^N_i \big\|^\theta_{H^s} \|\hat u^N_i\|^{1-\theta}_{H^s} \quad \text{for } \theta \in (0,1),
  \end{align*}
 where $s^{\prime}<s$. 
We then obtain 
\begin{equation}
\label{ho_1}
 \begin{aligned}
 \frac{d}{dt} \| \uh^{N}\|^2_{H^s}  + \tilde{\sigma}\big\| (-\Delta)^{\frac \alpha 2} \uh^{N}\big\|^2_{H^s} \lesssim  \big\|(-\Delta)^{\frac \alpha 2}\hat u^N \big\|^{1+\theta}_{H^s}   \|\hat u^N\|^{2-\theta}_{H^s}. 
 \end{aligned}
\end{equation}
 Integrating \eqref{ho_1} in time and applying H\"older's inequality gives that 
\begin{align*}
&  \| \uh^{N}(\tau)\|^2_{H^s} \hspace{-0.1 cm }  + \hspace{-0.1 cm }  \tilde{\sigma}\int_0^\tau  \hspace{-0.2 cm } \big\| (-\Delta)^{\frac \alpha 2} \uh^{N} \big\|^2_{H^s} \dd t 
\lesssim 
 \| u^0\|^2_{H^s} 
\hspace{-0.1 cm } + \hspace{-0.1 cm }\Big[ \hspace{-0.1 cm }\int_0^\tau \hspace{-0.2 cm }  \big\| (-\Delta)^{\frac \alpha 2} \uh^{N} \big\|^2_{H^s} \dd t \Big]^{\frac { 1+ \theta} 2} 
\Big[ \hspace{-0.1 cm } \int_0^\tau  \hspace{-0.2 cm } \|  \uh^{N}\|^{\frac{2(2-\theta)}{1-\theta}}_{H^s} \dd t \Big]^{\frac { 1- \theta} 2},
\end{align*}
which for $\tau \in (0, T]$ yields
 \begin{align*}
 \| \uh^{N}(\tau)\|^2_{H^s}  + \int_0^\tau \big\| (-\Delta)^{\frac \alpha 2} \uh^{N}\big\|^2_{H^s}  \dd t   \lesssim  \| u^0\|^2_{H^s}  + 
 \int_0^\tau  \|  \uh^{N}\|^{\frac{2(2-\theta)}{1-\theta}}_{H^s} \dd t . 
\end{align*}
An application of the generalized Gr\"{o}nwall inequality, see e.g.~\cite{Lip00}, and assumptions on the initial data  yield  \eqref{estim_Hs_1}. The relation \eqref{estim_hu_N} follows from Morrey's inequality.

We now give the argument for \eqref{estim_ho_1}. By the previous step, we use $\phi_i = D^l D^l_h \uh^N_i$ as a test function in  \eqref{def_sol:sys.reg}. Integrating by parts and taking the limit $h\to 0$  yields  
\begin{align*}
 & \frac{d}{dt} \big \|  D^l \uh^{N}_i \big \|_{2}^2 +  2\sigma_i \big\|  (-\Delta)^{\frac \alpha 2}  D^l  \uh_i^{N} \big\|^2_{2}
 \lesssim  
  \sum_{j=1}^n  \big\|(-\Delta)^{\frac{1-\alpha}2}   D^{l}\big(\uh_i^N\, \nabla^\beta\uh_j^{N} \ast \hat W_N \big)\big\|_2 \big\|  \nabla (-\Delta)^{\frac {\alpha-1} 2}   D^l \uh_i^{N}\big\|_{2},
\end{align*}
for all $l=0,\ldots, s$. 
 We then first apply the product rule to write
\begin{align*}
&  \big\|   (-\Delta)^{\frac{1-\alpha}2}  D^l \big(\uh_i^N \nabla^\beta \uh^{N}_j \ast \hat W_N \big)\big\|_{2} 
\leq  \sum_{m=1}^{l-1}  \big\|  (-\Delta)^{\frac{1-\alpha}2}  \big(D^{l-m} \uh_i^N  \, D^m \nabla^\beta \uh^{N}_j \ast \hat W_N \big)\big\|_{2} \\
& + \big\|  (-\Delta)^{\frac{1-\alpha}2}  \big(D^{l} \uh_i^N \,   \nabla^\beta \uh^N_j \ast \hat W_N\big)\big\|_{2}   +  \big\|  (-\Delta)^{\frac{1-\alpha}2}  \big( \uh_i^N  \, D^l \nabla^\beta \uh^{N}_j \ast \hat W_N\big)\big\|_{2} :={\rm J_1+J_2+J_3} .
\end{align*}
Applying the fractional Leibniz rule \eqref{frac_Leib_gen}, the last two terms on the right-hand side are estimated as
\begin{subequations} 
\begin{equation}
\label{May_23_3}
\begin{split}
{\rm J_2 }
& \lesssim   \big \|  (-\Delta)^{\frac{1-\alpha}2}  D^l \uh_i^N \,   \nabla^\beta \uh^{N}_j \ast \hat W_N \big\|_{2}
 + \big\| D^l \uh_i^N\,   (-\Delta)^{\frac{1-\alpha}2} \nabla^\beta \uh^{N}_j \ast \hat W_N \big\|_{2}  
 \\ 
 & \qquad + \big\| D^l (-\Delta)^{\frac{\alpha_1}2} \uh_i^N \big\|_{{p_1}} \big \| (-\Delta)^{\frac{\alpha_2}2} \nabla^\beta \uh^{N}_j \ast \hat W_N \big\|_{{p_2}}  := \rm J_{21}+  \rm J_{22} + \rm J_{23}, 
 \end{split}
\end{equation}
\begin{equation}
\label{May_23_2}
\begin{split}
{\rm J_3}
 & \lesssim   \big \|  (-\Delta)^{\frac{1-\alpha}2}  \uh_i^N\,   D^l \nabla^\beta \uh^{N}_j \ast \hat W_N \big\|_{2}
+ \big\|  \uh_i^N\,   (-\Delta)^{\frac{1-\alpha}2} D^l \nabla^\beta \uh^{N}_j \ast \hat W_N \big\|_{2}  
 \\ 
& \qquad + \big \| (-\Delta)^{\frac{\alpha_1}2} \uh_i^N \big\|_{{q_1}}  \big\| (-\Delta)^{\frac{\alpha_2}{2}} D^l \nabla^\beta \uh^{N}_j \ast \hat W_N \big\|_{{q_2}} :=  \rm J_{31} + \rm J_{32} + \rm J_{33}, 
 \end{split}
\end{equation}
\end{subequations}
where $\alpha_1 + \alpha_2 = 1- \alpha$ ($\alpha_1$ and $\alpha_2$ can be different in \eqref{May_23_3} and \eqref{May_23_2}) and $1/2 = 1/p_1 + 1/p_2 = 1/q_1 + 1/q_2$. 
We then apply H\"{o}lder's inequality and use Young's inequality for convolutions along with the $L^1$-normalization of $\hat{W}_N$ to write
\begin{equation} 
\label{May_23_1}
\begin{aligned} 
{\rm J_{21}} & \leq  \big \|  (-\Delta)^{\frac{1-\alpha}2}  D^{l} \uh_i^N  \big\|_{2p}  \big \| \nabla^\beta \uh^{N}_j  \big\|_{2p^\prime},   
&& \hspace{-0.2 cm }  {\rm J_{22}}  \leq  \big \|  D^l \uh_i^N  \big\|_{2\hat p}  \big \| (-\Delta)^{\frac{1-\alpha}2} \nabla^\beta \uh^{N}_j \big\|_{2\hat p^\prime} , \\
  {\rm J_{31}}  &\leq  \big \|   (-\Delta)^{\frac{1-\alpha}2}  \uh_i^N   \big\|_{2q}  \big \| D^l \nabla^\beta \uh^{N}_j  \big\|_{2q^\prime},  
  && \hspace{-0.2 cm }  {\rm J_{32}}  \leq  \big \|   \uh_i^N   \big\|_{2\hat q}  \big \|  (-\Delta)^{\frac{1-\alpha}2} D^l \nabla^\beta \uh^{N}_j   \big\|_{2\hat q^\prime},   
 \end{aligned}  
\end{equation}
where  $1/p + 1/p^\prime = 1/q + 1/q^\prime = 1/\hat p + 1/\hat p^\prime = 1/\hat q + 1/\hat q^\prime =1$.   We estimate each term separately and  split our arguments into two cases, which are $l=0$ and $1 \leq l \leq s$. 
\smallskip

\noindent{\textbf{Treatment of the $\rm J_{ij}$ for $i= 2,3$ and $j = 1,2,3$ when $l=0$}.} \quad \\
\noindent \textbf{(\textit{i})} $\rm J_{21}$ \textit{and} $\rm J_{31}$ \; Since $l=0$, we have that $\rm J_{21} = \rm J_{31}$. We then further distinguish between two cases:  $0<\alpha-\beta < d/2$ and $\alpha< {(d+2)}/4$. Notice that whenever $d>1$ the conditions are both trivially satisfied. 
\\
\noindent \textit{Case 1:}  $0<\alpha-\beta < d/2$.  Then there exists $\gamma \in (0, 1- \alpha)$ such that $d/(1-\beta - \gamma) > 2$.  We first notice that $p^{\prime}$ in \eqref{May_23_1} can be chosen such that $p^\prime >  d/(d- 2(1-\beta - \gamma))$. Then, using the theorem for Riesz potentials \eqref{Riesz_est}, that $\beta+ \gamma < 1$,  and the fractional Sobolev embedding \cite[Theorem~6.5]{NPV_2012}, yields 
\begin{align}\label{technical_1}
\begin{split}
 \big \| \nabla^\beta \uh^{N}_j  \big\|_{2p^\prime}   = \big \| (-\Delta)^{\frac{\beta + \gamma-1}2}  \nabla (-\Delta)^{- \frac{\gamma } 2} \uh^{N}_j  \big\|_{2p^\prime} \lesssim  \big \| \nabla (-\Delta)^{- \frac{\gamma + \alpha} 2}  (-\Delta)^{\frac \alpha 2}   \uh^{N}_j  \big\|_{r} \\   \lesssim  \big\|  \nabla (-\Delta)^{- \frac{\gamma + \alpha} 2}  (-\Delta)^{\frac \alpha 2}  \uh^{N}_j \big\|_{H^{s^\prime -1 + \alpha + \gamma}} 
 \lesssim \big \|(-\Delta)^{\frac \alpha 2}  \uh^{N}_j \big\|_{H^{s^\prime}},  
\end{split}
\end{align}
where $r= 2p^\prime d /(d + 2 p^\prime (1-\beta - \gamma)) > 2$ and $d/2+ \beta - \alpha \le  s^\prime  < s$. 
 By the Sobolev embedding we  obtain 
$$
\big \|  (-\Delta)^{\frac{1-\alpha}2}   \uh_i^N\,   \big\|_{2p}   \lesssim  \|\uh_i^N\,   \big\|_{H^s},
$$
for $1< p< d/(d - 2(s-1+ \alpha))$. To see that $1/p + 1/p^{\prime} =1$ is possible, notice that $p^{\prime} > d/(d- 2(1-\beta - \gamma))> d/(2(s-1+\alpha))$ since $s \geq d/2 $, $\gamma < 1 - \alpha$, and $\beta + 1 < 2 \alpha$.\\
\noindent \textit{Case 2:}  $0<\alpha< (d+2)/4$.  We choose $p$ in \eqref{May_23_1} such that $p > d/(d- 2(2\alpha -1))$. Then applying  \eqref{Riesz_est} and  the Sobolev embedding, implies 
\begin{align}
\label{technical_2.1}
\big \|  (-\Delta)^{\frac{1-\alpha}2}   \uh_i^N\,   \big\|_{2p} 
\lesssim  \big \|    (-\Delta)^{\frac{\alpha}2} \uh_i^N\,   \big\|_{r^\prime}   \lesssim  \big \|    (-\Delta)^{\frac{\alpha}2} \uh_i^N\,   \big\|_{H^{s^\prime}}, 
\end{align}
 where $r^\prime =2pd/(d+  2(2\alpha-1)p) > 2$ and $d/2 - (2\alpha -1) \le s^\prime < s$. We again apply the Sobolev embedding to write
$$
\begin{aligned} 
 \big \| \nabla^\beta \uh^{N}_j  \big\|_{2p^\prime}   =  \big \| (-\Delta)^{\frac{\beta -1}2}  \nabla \uh^{N}_j  \big\|_{2p^\prime}  \lesssim 
 \big \| (-\Delta)^{\frac{\beta -1}2}  \nabla \uh^{N}_j  \big\|_{H^{s- \beta}} \lesssim \big\|\uh_j^N\,   \big\|_{H^s}, 
  \end{aligned} 
$$
where we require that $1< p^\prime< d/(d- 2(s - \beta))$. Since $p\ge d/(d- 2(2\alpha -1))\ge d/(2(s-\beta))$ for $s\ge d/2$, the condition $1/p + 1/p^\prime = 1$ can be satisfied.\\
\noindent \textbf{(\textit{ii})} \textit{${\rm J_{22}}$ and $\rm J_{32}$}  \; Since $l=0$, we have that $\rm J_{22} = \rm J_{32}$. 

We first notice  
$ 
\|   \uh^{N}_i \|_{2 \hat p }  \lesssim  \| \uh_i^N \|_{H^s},
$
for  any $1<\hat  p < \infty$ since $s>d/2$. 
Under the conditions of both Case 1 or Case 2 above, we have  $\alpha < (d+2)/4 + \beta/2$. We now choose $0< \gamma < 1$ such that $2\alpha < \gamma + \beta + d/2 $ holds and set  $\hat{p}^{\prime}$ in~\eqref{May_23_1} such that $\hat p^\prime > d/(d-2(2\alpha - \beta - \gamma))$, to obtain 
\begin{equation*} 
\begin{aligned} 
& \big\| (-\Delta)^{\frac{1-\alpha}2} \nabla^\beta \uh^{N}_j \big\|_{2\hat p^\prime} = \big \| (-\Delta)^{\frac{\beta + \gamma -2\alpha}2}  (-\Delta)^{ - \frac \gamma 2 }   \nabla   (-\Delta)^{\frac \alpha 2 } \uh^{N}_j\big \|_{2\hat p^\prime} \\
 & \lesssim \big \|   (-\Delta)^{ - \frac \gamma 2 }   \nabla (-\Delta)^{\frac{\alpha}2}\uh^{N}_j\big \|_{r}  
 \lesssim \big \|   (-\Delta)^{ - \frac \gamma 2 }   \nabla (-\Delta)^{\frac{\alpha}2}\uh^{N}_j \big\|_{H^{s^\prime -1 + \gamma}}  \lesssim  \big\| (-\Delta)^{\frac{\alpha}2}\uh^{N}_j \big\|_{H^{s^\prime}}, 
 \end{aligned} 
\end{equation*}
where $r= 2\hat p ^\prime d / (d + 2 \hat p^\prime(2\alpha - \beta - \gamma))>2$ and  $ d/2 - (2\alpha - \beta -1)\le s^\prime < s$.\\ 
\noindent \textbf{(\textit{iii})} $\rm J_{23}$ \textit{and} $\rm J_{33}$ \quad Since $l=0$, we have that ${\rm J_{23}} = {\rm J_{33}}$. These terms can be estimated in a similar way as in~(\textit{ii}). In particular,  the Sobolev embedding yields 
\begin{equation*} 
 \big\| (-\Delta)^{\frac{\alpha_1}2} \uh_i^N \big\|_{{p_1}}  \le \|  \uh_i^N\|_{H^s}, 
\end{equation*}
where we require that $2< p_1 \le 2d/(d - 2(s-\alpha_1))$. 

We  notice that $\alpha < (d+2)/4 + \beta/2$ is satisfied in both Cases 1 and 2 above. Then we can fix $0< \gamma^\prime < 1$ such that $1+\alpha - \alpha_2 < \gamma^\prime + \beta + d/2$ for some $0<\alpha_2 < 1- \alpha$, where $\alpha_2$ is set in such a way that $\gamma^{\prime}$ exists. Furthermore, choosing $p_2$ such that $p_2 > 2d/(d - 2(1+ \alpha - \gamma^\prime - \alpha_2 - \beta))$, we can then estimate
\begin{equation*} 
\begin{aligned} 
 & \big\| (-\Delta)^{\frac{\alpha_2}2} \nabla^\beta \uh^{N}_j \big\|_{{p_2}}   =  \big\| \nabla  (-\Delta)^{-\frac{\gamma^\prime} 2} (-\Delta)^{\frac{\gamma^\prime + \alpha_2+ \beta - 1 - \alpha}2}  (-\Delta)^{\frac \alpha 2}  \uh^{N}_j \big\|_{{p_2}} \\
  &  \lesssim  \big\|   \nabla  (-\Delta)^{ - \frac {\gamma^\prime}  2 }   (-\Delta)^{\frac{\alpha}2}\uh^{N}_j \big\|_{r}    \lesssim  \big\|   \nabla   (-\Delta)^{ - \frac {\gamma^\prime}  2 }  (-\Delta)^{\frac{\alpha}2}\uh^{N}_j\big \|_{H^{s^\prime -1 + \gamma^\prime}}  \lesssim  \big\| (-\Delta)^{\frac{\alpha}2}\uh^{N}_j \big\|_{H^{s^\prime}}, 
  \end{aligned} 
\end{equation*}
where $r = p_2 d/ (d+ p_2 (1+ \alpha - \gamma^\prime - \alpha_2 - \beta))>2$ and $ d/2 - (\alpha - \alpha_2- \beta)\le s^\prime < s$. Since $s>d/2$ we have that $p_2 > 2d/(d - 2(1+ \alpha - \gamma^\prime - \alpha_2 - \beta)) \ge d/(s- \alpha_1)$, which implies that  $1/p_1 + 1/p_2 = 1/2$ can be satisfied.
\\
{\textbf{Treatment of the $\rm J_{ij}$ for $i= 2,3$ and $j = 1,2,3$ when $1 \leq l \leq s$}.}\quad \\
\noindent \textbf{(\textit{i})} ${ \rm J_{21}}$ \; We notice that
 $$
 \big \| \nabla^\beta \uh^{N}_j  \big\|_{2p^\prime}   =  \big \| (-\Delta)^{\frac{\beta -1}2}  \nabla \uh^{N}_j  \big\|_{2p^\prime}  \lesssim 
 \big \| (-\Delta)^{\frac{\beta -1}2}  \nabla \uh^{N}_j  \big\|_{H^{s- \beta}} \lesssim \big\|\uh_j^N\,   \big\|_{H^s},  
$$
 where we require that $1<p^\prime \leq d/(d- 2(s - \beta))$ if $s - \beta < d/2$ or any $1< p^{\prime} <\infty$ if $s- \beta \geq d/2$. Furthermore, we have the following embeddings:
\begin{align*}
&\big \| D^l (-\Delta)^{\frac{1-\alpha}2}   \uh_i^N\,   \big\|_{2p}  
\lesssim  \big \|   D^l (-\Delta)^{\frac{1-2\alpha}2}  (-\Delta)^{\frac{\alpha}2} \uh_i^N\,   \big\|_{H^{s^\prime - l -1 + 2\alpha } }   \lesssim  \big \|    (-\Delta)^{\frac{\alpha}2} \uh_i^N\,   \big\|_{H^{s^\prime}}, 
\end{align*}
 where  $d/(2(s- \beta))\leq p \leq d/ (d- 2 (s^\prime+ 2\alpha - s -1))$ and $s^{\prime}$ can be chosen to satisfy $\max\{ d/2 - (2\alpha - 1 - \beta), s+ 1 - 2\alpha \} < s^\prime <s$.     Notice that the lower bound for $p$ is derived from the upper bound for $p^{\prime}$, since the two are H\"{o}lder conjugates, and it is possible to choose $p$ and $p^{\prime}$ due to the restrictions on $s^{\prime}$.\\
\noindent \textbf{(\textit{ii})} ${\rm J_{31}}$ \; Using similar estimates as in the previous case, we obtain
$$
\big \|  (-\Delta)^{\frac{1-\alpha}2}   \uh_i^N\,   \big\|_{2q}     \lesssim  \big \|   \uh_i^N\,   \big\|_{H^{s}}
$$
where  we require that $1< q \le d/(d - 2(s+ \alpha -1))$ if $s+ \alpha -1 < d/2$ or $1< q <\infty$ if $s+ \alpha -1 \geq d/2$. Additionally, we find that 
$$
 \big \| D^l \nabla^\beta \uh^{N}_j  \big\|_{2q^\prime}   
 \lesssim 
 \big \|  D^l (-\Delta)^{\frac{\beta -1- \alpha}2}  \nabla   (-\Delta)^{\frac{\alpha}2}   \uh^{N}_j  \big\|_{H^{s^\prime- l -\beta + \alpha}} \lesssim \big\|  (-\Delta)^{\frac{\alpha}2}  \uh_j^N\,   \big\|_{H^{s^\prime}},  
$$
where $d/(2(s+\alpha -1)) \le q^\prime  \le d/(d - 2(s^\prime + \alpha - s- \beta))$. Again,  the lower bound for $q^{\prime}$ is derived from the upper bound for $q$. It is possible to choose an appropriate $q^{\prime}$ satisfying the above conditions if $\max\{d/2 - (2 \alpha -1 - \beta), s+ \beta - \alpha\} < s^\prime < s$.\\
\noindent \textbf{(\textit{iii})} ${\rm J_{22}}$ \quad  The Sobolev embedding ensures
$
 \big\| (-\Delta)^{\frac{1-\alpha}2} \nabla^\beta \uh^{N}_j \big\|_{2\hat p^\prime}  \lesssim  \| \uh^{N}_j \|_{H^{s}}, 
 $
where  $ 1 < \hat p^\prime \le d / (d - 2(s  + \alpha - 1 - \beta))$ if $s  + \alpha - 1 - \beta < d/2$ or $1< \hat{p}^{\prime} <\infty$ if $s+ \alpha - 1 - \beta \geq  d/2$. Additionally, we estimate 
\begin{equation*} 
\big\|  D^l  \uh^{N}_i \big\|_{2 \hat p } =   \big\|  D^l  ( -\Delta)^{- \frac \alpha 2}   ( -\Delta)^{\frac \alpha 2}   \uh^{N}_i \big\|_{2 \hat p }   \le  \big\|  ( -\Delta)^{\frac \alpha 2}  \uh_i^N \big\|_{H^{s^\prime} } 
\end{equation*}
for  $d/(2(s-(1+\beta - \alpha)))\le\hat  p \le  d/(d - 2(s^\prime +  \alpha - s))$. In order to ensure the existence of an appropriate $\hat  p$ we choose $s^{\prime}$ to satisfy $\max\{d/2 - (2\alpha - \beta -1),  s- \alpha\}<  s^\prime < s$.\\
\noindent \textbf{(\textit{iv})} ${\rm J_{32}}$ \quad  We have 
$ 
\|   \uh^{N}_i \|_{2 \hat q }   \le  \|   \uh_i^N \|_{H^s} 
$
for  any $1<\hat  q <\infty$,  since  $s\geq d/2$, and 
\begin{equation*} 
\begin{aligned} 
 \big\| D^l(-\Delta)^{\frac{1-\alpha}2} \nabla^\beta \uh^{N}_j \big\|_{2\hat q^\prime}  = 
 \big\| D^l \nabla (-\Delta)^{\frac{\beta-2\alpha}2}  (-\Delta)^{\frac \alpha 2}  \uh^{N}_j \big\|_{2\hat q^\prime} 
  \lesssim  \big\|  (-\Delta)^{\frac \alpha 2}  \uh^{N}_j \big\|_{H^{s^\prime}}, 
 \end{aligned} 
\end{equation*}
where we require that $ 1 < \hat q^\prime \le d / (d - 2(s^\prime + 2\alpha -1 - \beta- s)) $. It is possible to find such a $\hat q^\prime$ by setting $s^{\prime}$ to satisfy  $s- (2\alpha - 1- \beta) < s^\prime < s$. \\
\noindent \textbf{(\textit{v})} ${\rm J_{23}}$ \quad Using the Sobolev embedding, we find 
\begin{align*}  
{\rm J_{23} } \le \big\| D^{l} (-\Delta)^{\frac{\alpha_1 - \alpha}2}  (-\Delta)^{\frac \alpha 2}   \uh_i^N \big\|_{p_1} \big\| (-\Delta)^{\frac{\alpha_2}2} \nabla^\beta \uh^{N}_j  \big\|_{p_2}  \lesssim  \big \|  (-\Delta)^{\frac\alpha 2} \uh_i^N \big\|_{H^{s^\prime}}\|  \uh^{N}_j\|_{H^s},
\end{align*}
where   $2  < p_1 \leq 2 d/(d -2 (s^\prime + \alpha -s - \alpha_1))$ and $ d/(s^\prime+\alpha -s - \alpha_1)  \leq   p_2 \leq 2 d/(d-2(s - \beta - \alpha_2))$.  These relations are satisfied for $ \max\{ d/2 - (2\alpha  - \beta-1), s + \alpha_1 - \alpha\} < s^{\prime} < s$.\\
\noindent \textbf{(\textit{vi})} ${\rm J_{33}}$ \quad We estimate this term as
\begin{align*}  
{\rm J_{33} }  & \le  \big\| (-\Delta)^{\frac{\alpha_1}2} \uh_i^N  \big\|_{q_1}  \big\| D^l  \nabla  (-\Delta)^{\frac{\alpha_2+ \beta - 1 - \alpha}2}  (-\Delta)^{\frac{\alpha}2}  \uh^{N}_j  \big\|_{q_2}  \lesssim  \|\uh_i^N\|_{H^s}\big\|(-\Delta)^{\frac \alpha 2} \uh^{N}_j\big\|_{H^{s^\prime}}  ,
\end{align*}
where $2< q_1 \leq 2d /(d - 2(s - \alpha_1))$ if $s-\alpha_1 < d/2$  or $2<q_1<\infty$ if $s-\alpha_1 \geq d/2$ and, furthermore,  $d/(s- \alpha_1)\leq q_2 \leq 2d/(d - 2(s^\prime - s + \alpha_1 + 2\alpha -1 - \beta))$. These conditions stipulate that $s-s^\prime- (2\alpha - \beta-1)<\alpha_1$ and  $s^\prime \geq d/2 - (2\alpha -\beta -1)$. 

We remark that the  term ${\rm J_1}$ only appears when $l \geq 2$. Applying  \eqref{frac_Leib_gen} yields
$$
\begin{aligned}
  {\rm J_1}  
  \lesssim &   \sum_{m=1}^{l-1}     \big\|  (-\Delta)^{\frac{1-\alpha}2}   D^{l-m} \uh_i^N      D^m \nabla^\beta \uh^{N}_j \ast \hat W_N \big\|_{2} +    \big \| D^{l-m}  \uh_i^N    (-\Delta)^{\frac{1-\alpha}2}  D^m\nabla^\beta \uh^{N}_j \ast \hat W_N \big\|_{2} 
\\ &+  
  \sum_{m=1}^{l-1}  \big\| (-\Delta)^{\frac{\alpha_1}2} D^{l-m}  \uh_i^N \big\|_{p_1^m}  \big\| (-\Delta)^{\frac{\alpha_2}2} D^m \nabla^\beta \uh^{N}_j \ast \hat W_N \big\|_{p_2^m}  =:{\rm J_{11} + J_{12} + J_{13}}, 
\end{aligned}
$$
where $\alpha_1 + \alpha_2 = 1- \alpha$, $1/2 = 1/p_1^m + 1/p_2^m$.

 Since $m+ \beta <l\le s$ for all $m=1, \ldots, l-1$ and $2\le l \le s$,  we can estimate 
\begin{equation*} 
\begin{aligned}
{\rm J_{11}}  \le \sum_{m=1}^{l-1}  \big \|  (-\Delta)^{\frac{1-\alpha}2}  D^{l-m} \uh_i^N  \big\|_{{2q^\prime_m}}\big\| D^m \nabla^\beta \uh^{N}_j  \big\|_{{2q_m}}
 \lesssim   \big\|  (-\Delta)^{\frac{\alpha}2} \uh_i^N  \big\|_{H^{s^\prime}}    \| \uh^{N}_j \|_{H^s}, 
 \end{aligned}
\end{equation*}
 where  $q_m \leq  d/(d-2(s-m - \beta))$  if $s-m - \beta < d/2$ and $1< q_m < \infty$ if   $s-m - \beta \ge d/2$ and $d/(2(s-m-\beta))\leq q^\prime_m \leq d/(d-2(s^\prime -s + m -1 + 2 \alpha))$. These conditions stipulate that we choose $s^{\prime}$ to satisfy $\max\{ d/2 - (2\alpha -1 - \beta), s - (m-1) - 2\alpha, s - \alpha\} < s^\prime < s$. 
$$
{\rm J_{12}}
 \le \sum_{m=1}^{l-1}  \big \| D^{l-m} \uh_i^N  \big\|_{{2p_m}}  \big\| (-\Delta)^{\frac{1-\alpha}2}  D^m \nabla^\beta \uh^{N}_j   \big\|_{{2p^\prime_m}}
  \lesssim    \big\| \uh_i^N \big\|_{H^s}  \big \|  (-\Delta)^{\frac{\alpha}2} \uh^{N}_j  \big\|_{H^{s^\prime}},  
$$
 where  $1<p_m \leq d/(d-2m)$ if $m< d/2$ and $1< p_m< \infty$ if $m\ge d/2$ and $d/(2m) \le p_m^\prime \le  d/ ( d- 2(s^\prime - m+ 2\alpha -\beta - 1))$. This places the following condition on $s^{\prime}$:
 $\max\{ m+1 + \beta - 2 \alpha, d/2 - (2\alpha -1 - \beta), s- \alpha\} < s^\prime < s$, where $1\leq m \leq s-1$. \\
$$
  {\rm J_{13}}   \lesssim \| \uh^N_i \|_{H^s} \big\|(- \Delta)^{\frac \alpha 2}  \uh^{N}_j \big\|_{H^{s^\prime}}, 
$$
where $2<p_1^m\leq 2 d/(d- 2(m-\alpha_1)) $  and $d/(m-\alpha_1)\leq p_2^m\leq 2d/(d - 2(s^\prime -m + 2\alpha - \beta -1 + \alpha_1))$, which is satisfied if $\max \{d/2 - (2\alpha -\beta -1), s + \beta +1 - 2\alpha - \alpha_1\} < s^{\prime}<s$.

To conclude, we remark that combining all of the above estimates on ${\rm J_1}$,  ${\rm J_2}$, and  ${\rm J_3}$ and  summing over $l=1, \ldots, s$ and $i=1, \ldots, n$ yields \eqref{estim_ho_1}.

\smallskip

\noindent \textbf{Step 5: Global existence of solutions for \eqref{syst:reg} with small initial data.} In this step we show that there exists $\theta = \theta(d, \sigma_i, a_{ij}, n )$ such that if \eqref{small_ic} holds, then we can iterate the argument in Step 2 to obtain a global solution of  \eqref{syst:reg}. 

Using that $s^\prime <s$,  from  \eqref{estim_ho_1} we obtain 
\begin{equation}
\label{May_12_1}
 \begin{aligned}
\frac{d}{dt} \| \uh^{N}\|^2_{H^s}  + \tilde{\sigma} \big\| (-\Delta)^{\frac \alpha 2} \uh^{N} \big\|^2_{H^s} \le C(d, a_{ij}, n)  \big\|(-\Delta)^{\frac \alpha 2}\hat u^N \big\|^{2}_{H^s}   \|\hat u^N\|_{H^s},   
 \end{aligned}
\end{equation}
for some $\tilde{\sigma} >0$. With \eqref{May_12_1} in-hand, we can apply \cite[Lemma 17]{CDJ_2018} with 
\begin{align*}
f(t) = \| \hat u^N(t, \cdot)\|_{H^s}, \quad g(t) = \big\| (- \Delta)^{\frac{\alpha}{2}} \hat u^N(t,\cdot)\big\|_{H^s}, \quad a = \tilde{\sigma}, \quad \textrm{ and  } b = C(d, a_{ij}, n).
\end{align*} The lemma yields that if $\| u^0\|_{H_s} \leq a/b$ then $(d/dt) \|\hat{u}^N \|_{H^s}^2 \leq 0$ and, hence, $\| \hat u^N(t, \cdot )\|_{H^s} \leq a/b$ for any $t \in [0,T]$. Therefore, setting $\theta = a/b$ allows us to iterate the local existence result of Step~2 to obtain a global solution.

In the same way, we remark that using the higher-order regularity estimates and considering $\hat u^N \in L^2(0,T; H^{s + \alpha}(\R^d))^n$,  the time of existence of the local solutions from Step 2 can be made independent of $N$. 

To address the uniqueness and positivity of the solution, we remark that these properties can be shown in the same way as in Theorem \ref{full_system}.
\end{proof}

\section{Proof of Theorem~\ref{full_system}}
\label{ref_6}

\begin{proof} 
 In the first  step we use the uniformity in $N$ of the \textit{a priori} estimates \eqref{estim_Hs_1}  to pass to the  limit as $N \rightarrow \infty$, which yields a solution of \eqref{sys.final}. In the second step we show the non-negativity of solutions of \eqref{sys.final}. In the third step we prove the uniqueness of weak solutions of \eqref{sys.final}. To finish, in the fourth step, we prove strong convergence  of a sequence of solutions of  \eqref{syst:reg} to the solution of \eqref{sys.final}.  
\smallskip

\noindent {\bf Step 1:  Existence of solutions of \eqref{sys.final}.}
Since \eqref{estim_Hs_1} is uniform in $N$, by compactness there exists  $u \in L^\infty(0,T; H^s(\R^d))^n \cap L^2(0,T; H^{s+\alpha}(\R^d))^n$ so that 
$$
\begin{aligned} 
 \uh^{N}  & \rightharpoonup^\ast  u  && \text{ in }  L^\infty(0,T; H^s(\R^d))^n, \\
 \uh^{N}  & \rightharpoonup  u   && \text{ in }  L^2(0,T; H^{s+\alpha}(\R^d))^n,
 \end{aligned} 
 $$
where the $\uh^{N}$ are the solutions of \eqref{syst:reg} provided by Theorem \ref{existence_regular}. Furthermore, by \eqref{estim_Hs_1} and the lower semicontinuity of the norms we have that
 \begin{equation}\label{estim_u_11} 
\| u \|_{L^\infty(0,T; H^s(\R^d))}   + \| u \|_{L^2(0,T; H^{s+\alpha}(\R^d))}   \lesssim 1.
\end{equation}

We must still pass to the limit $N\rightarrow \infty$ in the weak formulation \eqref{def_sol:sys.reg}. 
We first notice that  
\begin{equation}\label{limit_2}
\hat{W}_N  \ast \nabla^{\beta} \uh_j^N \rightharpoonup \nabla^{\beta} u_j \quad \text{ in }  L^2(0,T; L^2(\R^d)),
\end{equation}
which follows, e.g.~from \eqref{converg_convol}. Furthermore, using the equation \eqref{syst:reg}, we remark that 
\begin{equation}\label{time_deriv}
\| \partial_t \uh^N \|_{L^2(0,T;H^{\alpha}(\R^d)^{\prime})}  \lesssim 1,
\end{equation} 
where we have used \eqref{estim_Hs_1} and Morrey's inequality. 

Now, for any $R>0$, since the embedding of $H^{\alpha}(B_R)$ into $L^2(B_R)$ is compact and by~\eqref{time_deriv}, the Aubin-Lions lemma yields  $\uh_i^N \rightarrow u_i$ strongly in $L^2(0,T; L^2(B_R))$. To finish, we consider the weak formulation \eqref{def_sol:sys.reg} for a test function $\psi \in C^{\infty}_0(0,T; C^{\infty}_0 (\R^d))$. Using the observations made above, we are then able to pass to the limit in the nonlinear term of \eqref{def_sol:sys.reg}. Standard arguments ensure that the initial condition is satisfied.

\smallskip 

\noindent {\bf Step 2:  Positivity of solutions of \eqref{sys.final}.} Considering $u_i^{-} = \min\{ u_i, 0\}$  as a test function in the weak formulation of  \eqref{sys.final}, we then obtain 
\begin{equation}
\label{A_24_1}
 \frac d{dt} \| u_i^{-} \|_2^2 + 2 \sigma_i \big\langle  u_i^{-} , (-\Delta)^{\alpha} u_i  \big\rangle
\le C_\varsigma\sum_{j=1}^n  \big\|(-\Delta)^{\frac{1-\alpha}{2}}(u_i^{-}\nabla^{\beta} u_j )\big\|^2_2 + \varsigma \big\| (-\Delta)^{\frac \alpha 2} u_i^{-}  \big\|^2_2  
\end{equation}
for  $t \in (0,T]$ and $\varsigma>0$. To treat the second term on the left-hand side of \eqref{A_24_1} we use \cite[Lemma 5.2]{PGRV_2012}. In particular, we find that for any $t \in (0,T]$, the relation 
\begin{align*}
 \big\langle  u_i^{-} , (-\Delta)^{ \alpha } u_i  \big\rangle  \geq \int_{\R^d} |(-\Delta)^{\frac \alpha 2} u_i^{-}|^2 \dd x
\end{align*}
holds. Combining this observation with \eqref{A_24_1} and using \eqref{frac_Leib_gen}, we find that
\begin{align*}
\begin{split}
&  \frac d{dt} \| u_i^{-} \|_2^2  + \big\| (-\Delta)^{\frac{\alpha}{2}} u_i^- \big\|^2_2   \lesssim   \sum_{j=1}^n \big\|(-\Delta)^{\frac{\alpha_1}2}\nabla^\beta u_j \big\|^2_{p_1} \big\|(-\Delta)^{\frac{\alpha_2}2} u_i^{-} \big\|^2_{p_2} \\
 & + \sum_{j=1}^n \Big[ \|\nabla^\beta u_j\|^2_{L^\infty} \big\|(-\Delta)^{\frac{1-\alpha} 2} u_i^{-} \big\|^2_2 + 
\big\| (-\Delta)^{\frac{1-\alpha} 2} \nabla^\beta u_j \big\|^2_{L^\infty} \|u_i^{-}\|^2_2  \Big], 
\end{split}
 \end{align*}
where $1/ p_1 + 1/p_2 = 1/2$ and $\alpha_1 + \alpha_2 = 1-\alpha$. This we then combine with the observation  
\begin{align*}
 & \big\|(-\Delta)^{\frac{\alpha_1}2}\nabla^\beta u_j\big\|^2_{p_1} \big\|(-\Delta)^{\frac{\alpha_2}2} u_i^{-}\big\|^2_{p_2}  \lesssim
 \|u_j\|^2_{H^{s+ \beta}} \| u_i^{-}\|^2_{H^{1-\alpha}}, 
\end{align*}
where we require that $2<p_1\leq 2d/(d-2(s - \alpha_1))$ and $d/(s - \alpha_1) \leq p_2 \leq 2d/(d-2\alpha_1)$. We are able to satisfy these conditions since $s> d/2$.  Plugging-in this relation, using the embedding of $ H^{s} \hookrightarrow L^{\infty}$ for $s>d/2$, and additionally using the fractional Gagliardo-Nirenberg interpolation inequality, we obtain 
\begin{equation*}
 \frac{d}{dt} \|u^{-}\|^2_{2} 
 \lesssim  \Big( \|u\|^{\frac{2\beta}{2\alpha-1}}_{H^{s+ \alpha}} \|u\|^{\frac{2(\alpha - \beta)}{2\alpha-1}}_{H^{s}} + \|u\|^{2}_{H^{s+ \beta}}
+  \|u\|^{2\frac{2\alpha- 1-\beta}\alpha}_{H^{s}}\|u\|^{2\frac{1+\beta - \alpha}\alpha}_{H^{s+ \alpha}} \Big) 
 \|u^{-}\|^{2}_{2}. 
  \end{equation*}
From the above estimate and using the regularity of $u$ and non-negativity of initial data, we conclude that $u_i \ge 0 $ in $(0,T)\times \R^d$, for $i=1, \ldots, n$.

\smallskip

\noindent {\bf Step 3:  Uniqueness  of solutions of \eqref{sys.final}.} We assume that there are two solutions $u^1$ and $u^2$ of \eqref{sys.final}  and consider $w_i=u^1_i- u^2_i$ as a test function in the weak formulation of the equation for  $w_i$:
\begin{equation} 
\label{May_23_4}
\begin{aligned}
& \sum_{i=1}^n \frac{d}{dt}  \|w_i\|^2_2 +    \big\|(-\Delta)^{\frac \alpha 2} w_i\big\|^2_2  \lesssim 
\sum_{i, j=1}^n \big\|(-\Delta)^{\frac{1-\alpha} 2} \big(w_i \nabla^\beta u_j^1 +  u_i^2 \nabla^\beta w_j\big)\big\|^2_2 
\\ 
& \quad    \lesssim  \sum_{i, j=1}^n \| w_i\|^{2\theta}_{H^{\alpha}} \| w_i\|^{2(1-\theta)}_2   \|u_j^1\|^2_{H^{s+ \beta + 1- \alpha}} 
+ \|u_i^2\|^2_{H^{s+1 -\alpha}} \|w_j \|^{2\theta_1}_{H^{\alpha}}\|w_j \|^{2(1-\theta_1)}_{2}, 
  \end{aligned}
\end{equation} 
where $\theta = 1/\alpha -1$ and $\theta_1 = (1+\beta - \alpha)/\alpha$ are defined by applying Gagliardo-Nierenberg inequality. 
To obtain \eqref{May_23_4}, we have used \eqref{frac_Leib_gen} and 
\begin{equation*} 
\begin{aligned} 
 \big\| (- \Delta)^{\frac{\alpha_1}2} w_i \big\|^2_{p_1} \big\| (- \Delta)^{\frac{\alpha_2}2} \nabla^\beta  u_j^1 \big\|^2_{p_2}
   \lesssim   \| w_i\|^2_{H^{1-\alpha}} \|u_j^1\|^2_{H^{s+ \beta + 1- \alpha}}, \\
   \big \| (- \Delta)^{\frac{\alpha_1}2} u^2_i \big\|^2_{\hat{p}_1} \big \| (- \Delta)^{\frac{\alpha_2}2} \nabla^\beta  w_j \big\|^2_{\hat{p}_2}
   \lesssim   \| u^2_i\|^2_{H^{s+1-\alpha}} \|w_j\|^2_{H^{\beta + 1- \alpha}},  
 \end{aligned}   
\end{equation*} 
where  $2<p_1\leq 2d/(d - 2\alpha_2)$, $2<\hat{p}_2\leq 2d/(d - 2\alpha_1)$, $s\geq d/2 - (1-\alpha)$, and $\alpha_1 + \alpha_2 = 1 - \alpha$. 
Integrating \eqref{May_23_4} in time and applying Young's inequality gives
\begin{equation*} 
\begin{aligned}
  &\|w(\tau)\|^2_2  +  \int_0^\tau \big\|(-\Delta)^{\frac \alpha 2} w (t) \big \|^2_2 \,  \dd t  \lesssim  \int_0^\tau  \|w (t) \|^{2}_2 \Big(\|u^2(t)\|^{\frac{2}{1-\theta_1}}_{H^{s+1 -\alpha}} + 
 \|u^2(t)\|^{2}_{H^{s+1 -\alpha}}  
 \\ & \hspace{6cm} +  \|u^1(t)\|^{\frac 2{1-\theta}}_{H^{s+ \beta + 1- \alpha}}+  \|u^1(t)\|^2_{H^{s+ \beta + 1- \alpha}}\Big)  \dd t, 
   \end{aligned}
\end{equation*} 
for any $\tau \in (0,T]$. An application of the Gr\"onwall inequality implies that $w_i = u_i^1 - u_i^2 = 0$ a.e.\ in $(0,T)\times \R^d$, and hence uniqueness of a solution of    \eqref{sys.final}.
\smallskip 

\noindent {\bf Step 4:  Strong convergence   of   a sequence of solutions of  \eqref{syst:reg}  to solution of \eqref{sys.final}. }
Finally, we prove the strong convergence of $\hat u^N$ to $u$. We consider  the equation for $\hat u_i^N- u_i$ in the weak form, and use as a test function $\hat u_i^N-u_i$ to obtain
\begin{equation*} 
\begin{aligned}
& \frac d{dt} \|u_i-\hat u_i^N\|^2_2  +   \big\|(-\Delta)^{\frac \alpha 2} (u_i-\hat u_i^N) \big\|^2_2
\lesssim \sum_{j=1}^n  \Big[\big\|(-\Delta)^{\frac{1-\alpha}{2}}\big[(u_i-\hat u_i^N)\na^{\beta}\hat u_j^N \ast \hat W_N \big]\big\|^2_{2}\\
&+ 
 \big\| (-\Delta)^{\frac{1-\alpha}{2}}\big(u_i\, \na^{\beta} (\hat u_j^N  -  u_j)\big)\big\|^2_{2} + \big\| (-\Delta)^{\frac{1-\alpha}{2}}\big(u_i[ \na^{\beta}(\hat u_j^N \ast \hat W_N)-\na^{\beta} \hat u^N_j]\big)\big\|^2_{2}
  \Big]  =: {\rm J_1 + J_2 + J_3}.  
 \end{aligned}
\end{equation*}
 We  estimate the terms on the right hand-side using \eqref{frac_Leib_gen} and for the first term obtain 
\begin{equation*} 
\begin{aligned}
{\rm J_1} &\lesssim   \sum_{j=1}^n \Big[ \big\|\nabla^\beta \hat u_j^N \big\|^2_{L^\infty} \big\|(-\Delta)^{\frac{1-\alpha} 2} (u_i - \hat u_i^N)\big\|^2_2  + 
\big\| (-\Delta)^{\frac{1-\alpha} 2} \nabla^\beta \hat u_j^N \big\|^2_{L^\infty} \|u_i - \hat u_i^N\|^2_2  \Big] 
\\
& \quad +  \sum_{j=1}^n \big\|(-\Delta)^{\frac{\alpha_1}2}\nabla^\beta\hat  u^N_j \big\|^2_{p_1}  \big\|(-\Delta)^{\frac{\alpha_2}2} (u_i -\hat u_i^N) \big\|^2_{p_2}  
\\ 
& \lesssim  \sum_{j=1}^n \|\hat u^N_j\|^2_{H^{s+ \beta}}  \|u_i - \hat u^N_i\|^{2\theta}_{H^{\alpha}} \|u_i - \hat u_i^N\|^{2(1-\theta)}_2 + 
\| \hat u^N_j\|^{2\theta_1}_{H^{s+ \alpha}}\|\hat  u_j^N\|^{2(1-\theta_1)}_{H^{s}}   \|u_i - \hat u_i^N\|^2_2
\\
& \qquad \qquad  +  \|\hat u_j^N\|^2_{H^{s+ \beta}} \Big[ \| u_i - \hat u_i^N\|^{2}_2  + \big\| (-\Delta)^{\frac{\alpha} 2} (u_i- \hat u_i^N)\big\|^{2\theta}_2 \| u_i - \hat u_i^N\|^{2(1-\theta)}_2   \Big], 
 \end{aligned}
\end{equation*}
where $\theta = 1/\alpha - 1$, $\theta_1 = (1+\beta - \alpha)/\alpha$  and  $\alpha_1 + \alpha_2 = 1-\alpha$.
For the second term in the similar  way we find that 
\begin{equation*} 
\begin{aligned}
& {\rm J_2}  \lesssim  
\sum_{j=1}^n \Big[  \| u_i\|^2_{H^{s+ 1-\alpha}}  \|u_j - \hat u^N_j\|^{2\theta_2}_{H^{\alpha}} \|u_j - \hat u_j^N\|^{2(1-\theta_2)}_2  
\\ &  \quad + 
\|  u_i\|^{2}_{H^{s}}\| u_j -  \hat u_j^N\|^{2\theta_1}_{H^{ \alpha}}   \|u_j - \hat u_j^N\|^{2(1-\theta_1)}_{2}
 +  \| u_i\|^2_{H^{s}} \| u_j- \hat u_j^N\|^{2\theta_1}_{H^\alpha} \| u_j  - \hat u_j^N\|^{2(1-\theta_1)}_{2} \Big], 
 \end{aligned}
\end{equation*}
where $\theta_1 = (1+\beta - \alpha)/\alpha$ and $\theta_2 = \beta/\alpha < 1$. 
For the third term  we have 
\begin{equation*} 
\begin{aligned}
{\rm J_3}  & \lesssim
   \| u_i\|^2_{H^{s+1-\alpha}}  \big \| \na^{\beta} \hat u_j^N \ast \hat W_N - \na^{\beta} \hat u^N_j \big\|_{2}^2  + \| u_i\|^2_{H^{s}}   \big\| \na^{\beta} \hat u_j^N \ast \hat W_N - \na^{\beta} \hat u^N_j \big\|_{H^{1-\alpha}}^2\\
 & \quad  + \| u_i\|^2_{H^{s}}   \big\| (-\Delta)^{1-\alpha} (\na^{\beta} \hat u_j^N \ast \hat W_N -  \na^{\beta} \hat u^N_j) \big\|_{2}^2 
 \lesssim  \hat \kappa_N^{-2} \| u_i\|^2_{H^{s+1-\alpha}}\| \hat u_j^N\|^2_{H^{\beta + 2 - \alpha }}  , 
 \end{aligned}
\end{equation*}
 see \cite{O_89} or estimate \eqref{converg_convol} of Lemma~\ref{conv_2}. 
Then applying Young's inequality yields
\begin{equation*} 
\begin{aligned}
 \sum_{i=1}^n \frac{d}{dt} \|u_i - \hat u_i^N\|^2_2  + \sum_{i=1}^n \|(-\Delta)^{\frac \alpha 2} (u_i - \hat u_i^N)\|^2_2
 \lesssim  \sum_{i,j=1}^n \Big[ \Big(1+ \|\hat u^N_j\|^{\frac{2\beta}{2\alpha-1}}_{H^{s+ \alpha}} \|\hat u^N_j\|^{\frac{2(\alpha - \beta)}{2\alpha-1}}_{H^{s}} 
\\  \quad  +  \|\hat u^N_j\|^{\frac{2(2\alpha- 1-\beta)}{\alpha}}_{H^{s}}\|\hat u^N_j\|^{\frac{2(1+\beta - \alpha) }{\alpha}}_{H^{s+ \alpha}}  
 +  \|u_j\|^{\frac{2\alpha}{2\alpha-1-\beta}}_{H^{s}} +   \|u_j\|^{\frac{2(1-\alpha)}{\alpha-\beta}}_{H^{s+ \alpha}} \|u_j\|^{\frac{2(2\alpha-1)}{\alpha-\beta}}_{H^{s}}   \Big) 
 \|u_i - \hat u_i^N\|^{2}_{2}  \\
\quad   + \hat \kappa_N^{-2} \| u_i\|^2_{H^{s+1-\alpha}}\| \hat u_j^N\|^2_{H^{\beta + 2 - \alpha}} \Big] .
  \end{aligned}
\end{equation*}
Using the regularity of $u$ and $\hat u_N$, the definition of $\hat{\kappa}_N$,  and applying the Gr\"onwall inequality, we obtain the convergence  result in \eqref{converg_hatu_u}. 
 \end{proof} 
 
  \section*{Acknowledgements}
  ESD and CR gratefully acknowledge partial support from the Austrian Science Fund (FWF), grants P30000, W1245, and F65.
 

\section*{Appendix}
We now summarize some facts about fractional Sobolev spaces and the fractional Laplacian that we use throughout the paper. For a more complete picture  see \cite{NPV_2012} and \cite{S_2019}. 

\begin{definition}[Fractional Sobolev norm $H^{\alpha}(\R^d)$] \label{seminorm} Let $\alpha \in (0,1)$. We define the fractional $H^{\alpha}$-seminorm as  
$$
[\psi]^2_{H^{\alpha}} := \int_{\R^d} \int_{\R^d}   \frac{|\psi(x) - \psi(y)|^2}{|x-y|^{d+2\alpha}} \dd x \dd y
$$
and remark that the $H^{\alpha}$-norm is then given by 
$
\|\psi\|^2_{H^{\alpha}} := \|\psi\|_{2}^2 + [\psi]^2_{H^{\alpha}}$.
\end{definition}
\noindent  The other fractional Sobolev spaces are defined analogously, see e.g.~\cite[Section 2]{NPV_2012}. Throughout the article the following equivalences are used:
\begin{align}
\label{equivalence}
\| \nabla (-\Delta)^{\frac{ \alpha - 1} 2} \psi \|_{2} ^2   \sim \|(-\Delta)^{\frac \alpha 2} \psi \|^2_{2}, \qquad \| (-\Delta)^{\frac{\alpha}{2}} \psi \|_{2} \sim [\psi]_{H^{\alpha}}
\end{align}
and  can be found in \cite[Prop.~3.6]{NPV_2012}. These are simple consequences of the Fourier analytic definition of the fractional Laplacian. 

 For $f\in H^1(\R^d)$ and $g \in H^{1-\alpha}(\R^d)$, with $\alpha \in (0,1)$,  it holds that    
\begin{align}
\label{integration_by_parts}
\langle \na f ,  g \rangle = \langle \nabla (- \Delta )^{(1 - \alpha)/2} (- \Delta )^{(\alpha -1)/2}  f , g  \rangle  = \langle \nabla  (- \Delta )^{(\alpha -1)/2}  f , (- \Delta )^{(1 - \alpha)/2} g  \rangle.
\end{align}
Furthermore, for the fractional Laplacian the classical product rule may be replaced by the following commutator estimate: 
\begin{align}
\label{frac_Leib_gen}
\big\| (-\Delta)^{\frac \alpha 2} (fg) -  (g (-\Delta)^{\frac \alpha 2} f  + f (-\Delta)^{\frac \alpha 2} g)  \big\|_{p} \lesssim \big\| (-\Delta)^{\frac{\alpha_1}2} f \big\|_{p_1} \big\| (-\Delta)^{\frac{\alpha_2}2} g \big\|_{p_2}  ,
\end{align}
where $1/p = 1/{p_1} + 1/{p_2}$ with $p_1, p_2 \in (1, \infty)$ and $\alpha = \alpha_1 + \alpha_2$ with $\alpha_1, \alpha_2 >0$, see \cite{KPV_93}. We often make use of \eqref{frac_Leib_gen} in the form
\begin{align}
\label{frac_Leib}
\begin{split}
\big\| (-\Delta)^{\frac{\alpha}{2}} (fg) \big\|_{2} \lesssim  \| g \|_{H^{\alpha +s}}  \|f\|_2 + \|g\|_{H^{\alpha + s^{\prime}}} \|  f \|_{H^\alpha} \, \,  \textrm{for} \,\,  s \geq \frac{d}{2} \, \, \textrm{ and } \, \, s^{\prime}\geq \frac{d}{2} - \alpha.
\end{split}
\end{align}
We remark that \eqref{frac_Leib} is a simple consequence of \eqref{equivalence}, \eqref{frac_Leib_gen}, and the Sobolev embedding for fractional Sobolev spaces, which can be found in \cite[Theorem 6.5]{NPV_2012}.

We will also make use of the estimate 
\begin{align}\label{frac_Leib_2}
\big\| (-\Delta)^{\frac{\alpha}{2}} (fg)\big\|_{2}
 \lesssim  \big(  \|g\|_{\infty} +   \|\nabla g\|_{\infty } \big) \| f \|_{H^\alpha},
\end{align}
which holds for  $g \in W^{1,\infty}(\R^d)$, $f \in H^{\alpha}(\R^d)$.  To show \eqref{frac_Leib_2}
we  use  \eqref{frac.L}  and  obtain 
$$
 \int\limits_{\R^d} \big|(-\Delta)^{\frac{\alpha}{2}} (fg)(x)\big|^2 \dd x   \leq  \big\|  g(-\Delta)^{\frac{\alpha}{2}}  f \big \|_2^2 
 + \int\limits_{\R^d} \Big| \textrm{P.V.} \int\limits_{\R^d} \frac{g(x) - g(y)}{|x-y|^{d+\alpha}} f(y) \dd y \Big|^2\dd x
  = {\rm J_1 + J_2}, 
$$
 where we  bound
$
{\rm J_1}  \lesssim  \|g\|^2_{L^\infty} \|(-\Delta)^{\frac{\alpha}{2}}  f \|^2_2
$
and decompose $\rm{J}_2$ as
\begin{align*}
&{\rm J_{21} + J_{22}}:= \int\limits_{\R^d}\Big|  \textrm{P.V.} \hspace{-0.3 cm } \int\limits_{|x-y|<1}  \hspace{-0.2 cm }   \frac{g(x) - g(y)}{|x-y|^{d+\alpha}} f(y) \, \dd y \Big|^2\dd x+ 
 \int\limits_{\R^d}\Big | \int\limits_{|x-y|\geq 1}  \hspace{-0.2 cm }   \frac{g(x) - g(y)}{|x-y|^{d+\alpha}} f(y) \, \dd y \Big|^2\dd x.
\end{align*}
Considering the following $L^1$-functions
   $$
 h_1(x) =\begin{cases}
\dfrac 1{ |x|^{d+\alpha-1} } & \text{ for } |x|< 1, \\
 0 & \text{ otherwise},  
 \end{cases}  \quad \text{ and  } \quad 
  h_2(x) =\begin{cases}
\dfrac 1 { |x|^{d+\alpha} }  & \text{ for } |x|\geq 1, \\
 0 & \text{ otherwise}, 
 \end{cases} 
 $$   
 we obtain 
 $$ 
\begin{aligned}
J_{21} & \lesssim   \| \nabla g \|_{L^\infty}^2   
\int_{\R^d} \Big( \int_{|x-y|< 1}   \frac{|f(y)|}  { |x-y |^{d + \alpha -1}}  \dd y\Big)^2 \dd x  
 \lesssim   \| \nabla g \|_{L^\infty}^2   \| f \|^2_{2} , \\
 J_{22} & \lesssim   \|  g \|_{L^\infty}^2 
\int_{\R^d} \Big( \int_{|x-y|\geq 1}   \frac{|f(y)|}  { |x-y |^{d + \alpha}}    \dd y\Big)^2 \dd x
\lesssim   \|  g \|_{L^\infty}^2  \| f \|^2_{2}.  
\end{aligned}
$$      
The Gagliardo-Nirenberg  inequality for fractional Sobolev spaces (see e.g.~\cite{Brezis_2018})
reads  
\begin{align}
\label{GN_form}
\| f \|_{W^{s, p}(\R^d)} \lesssim \|f \||_{W^{s_1, p_1}(\R^d)}^\theta  \| f \||_{W^{s_2, p_2}(\R^d)}^{1-\theta}
\end{align}
for $s= \theta s_1 + (1- \theta) s_2$ and  $1/p = \theta/p_1 + (1- \theta)/p_2$, where $0\leq s_1, s_2$, $1\leq p_1, p_2 \leq \infty$, and $\theta \in (0,1)$.

To finish, we remark that for the inverse fractional Laplace operator we have 
$$
(-\Delta)^{-\kappa} f(x) = \frac 1 {c_{d, \kappa}} \int_{\R^d} \frac{ f(y)}{|x-y|^{d-2\kappa}} \dd y =  \mathcal I_{2\kappa} \ast f  (x), \; \;  \; \; \; 
 \mathcal I_{2\kappa}(x) = \frac 1 {c_{d, \kappa}}   |x|^{-(d-2\kappa)}, 
$$
for $d> 2\kappa>0$,  and  for $p< d/(2\kappa)$, see e.g.~\cite[Chapter 5, Theorem 1]{Stein}, 
\begin{equation}\label{Riesz_est} 
\|  (-\Delta)^{-\kappa} f \|_{L^{dp/(d-2\kappa p)}(\R^d)} \lesssim \| f\|_{L^p(\R^d)}.
\end{equation}

We now show that the limit $N\to \infty$ of the stopping time $t^N$ defined in \eqref{stopping_time} is positive a.s.~in $\Omega$. Towards a contradiction let us assume that $t^N \to 0$ with a positive probability, i.e.~$\mathbb P(\omega\in \Omega: \; t_N(\omega) \to 0 ) \geq \varepsilon_0>0$. For $\omega \in \Omega$ s.t.~$t^N(\omega) \rightarrow 0$, for any $\tilde \kappa>0$ there exists $N_0(\omega)$ such that $t^{N} < \tilde \kappa$ for all $N \geq N_0$. We remark that by Egoroff's theorem the $N_0$ can be chosen uniformly in $\omega$ on a set of measure $\varepsilon_0/2$. Letting $ \tilde \kappa < T$, we thereby obtain that 
$$
\mathbb P\Big(\|h^{N} - \hat u^{N}\|^2_{[0,t^{N}]} \geq \delta_{N} \Big) \geq \frac{\varepsilon_0}{2}, \quad \text{for } N \geq N_0 \text{ and } N_0 \gg 1. 
$$
On the other hand, \eqref{initial_condition_assumption_1} and \eqref{initial_condition_assumption_2} ensure that for any $\varepsilon>0$ there exists $\tilde N$ such that
$$
\mathbb P \big(\|h^N(0) - \hat u^N(0)\|^2_2\geq \delta_N^{1+\rho} \big) < \varepsilon, \quad \text{for } N\geq \tilde N.
$$
From the right-continuity of $\|h^N - \hat u^N\|^2_{[0,\tau]}$  and since $\|h^N - \hat u^N\|^2_{[0,\tau_m]}$ is monotone non-increasing as $\tau_m\searrow 0$,  we obtain that  there exists $\tau_{M}>0$ such that 
$$
\mathbb P \Big( \|h^{N} - \hat u^{N}\|^2_{[0,\tau]}\geq \delta_{N} \Big) < 2\varepsilon, \quad \text{for } \tau \leq \tau_{M} \text{ and } N \geq \tilde N.
$$ 
Then taking $\varepsilon < \varepsilon_0/6$, $\tilde \kappa = \tau_M$, and $N_0\geq\tilde N$, we obtain a contradiction
\begin{align*}
\mathbb P \Big(\|h^{N} - \hat u^{N}\|_{[0,t^{N}]}^2\geq \delta_{N} \Big) & \leq \mathbb P \Big(\|h^{N} - \hat u^{N}\|_{[0,\tau_M]}^2\geq \delta_{N} \Big) < 2\varepsilon, \quad \text{for } N\geq N_0.
\end{align*}

\bibliographystyle{plain} 
\bibliography{M_E.bib}

\end{document}